\documentclass[11pt]{article}
\usepackage{hyperref}
\usepackage{mathptmx}
\usepackage{amsmath,amssymb,amsthm}
\usepackage{authblk}
\usepackage[final]{graphicx}
\usepackage{subfig}
\usepackage{color}
\usepackage[margin=0.75in]{geometry}
\usepackage{chngcntr}
\usepackage{indentfirst}
\usepackage{enumitem}
\usepackage{hyperref}
\usepackage[abspath]{currfile}
\usepackage{float}
\usepackage{bbm}
\usepackage{natbib}
\usepackage{siunitx}
\usepackage{mathrsfs}
\usepackage{xcolor}
\usepackage[linesnumbered,ruled,vlined]{algorithm2e}

\SetCommentSty{mycommfont}

\SetKwInput{KwInput}{Input}                
\SetKwInput{KwOutput}{Output}              
\SetKw{Continue}{continue}

\usepackage{optidef}
\usepackage{cleveref}
\usepackage{enumitem}

\usepackage{apxproof}

\DeclareMathOperator*{\argmin}{arg\,min}

\usepackage{tikz}
\usetikzlibrary{matrix,positioning}
\tikzset{bullet/.style={circle,fill,inner sep=2pt}}
\usepackage{adjustbox}
\usepackage{pgfplots}

\title{ Tight Bounds for a Class of Data-Driven \\ Distributionally Robust Risk Measures}
\author{Derek Singh, \, Shuzhong Zhang}
\affil{Department of Industrial and Systems Engineering, University of Minnesota \\ singh644@umn.edu, \, zhangs@umn.edu}
\date{\vspace{-5ex}}

\theoremstyle{case}
\newtheorem{case}{Case}

\newtheorem{remark}{Remark}
\newtheoremrep{prop}{Proposition}[section]
\newtheoremrep{theorem}{Theorem}[section]
\newtheoremrep{lemma}[theorem]{Lemma}
\newtheoremrep{conj}[theorem]{Conjecture}
\newtheoremrep{corollary}{Corollary}[theorem]

\newtheorem*{theorem*}{Theorem}
\newtheorem*{prop*}{Proposition}

\providecommand{\keywords}[1]
{
  \small	
  \textbf{\textit{Keywords---}} #1
}

\usepackage[section]{placeins}

\DeclareMathAlphabet{\mathcal}{OMS}{cmsy}{m}{n}

\begin{document}

\maketitle

\begin{abstract}
This paper expands the notion of robust moment problems to incorporate distributional ambiguity using Wasserstein distance as the ambiguity measure. The classical Chebyshev-Cantelli (zeroth partial moment) inequalities, Scarf and Lo (first partial moment) bounds, and semideviation (second partial moment) in one dimension are investigated. The infinite dimensional primal problems are formulated and the simpler finite dimensional dual problems are derived. A principal motivating question is how does data-driven distributional ambiguity affect the moment bounds. Towards answering this question, some theory is developed and computational experiments are conducted for specific problem instances in inventory control and portfolio management. Finally some open questions and suggestions for future research are discussed.
\end{abstract}

\keywords{robust moment problems, Chebyshev-Cantelli inequality, Scarf and Lo bounds, partial moments, Wasserstein distance, Lagrangian duality}

\section{Introduction and Overview}
\subsection{Moment Problems}
\indent An important class of moment problems involves computing bounds for certain quantities such as shortfall probability, lost sales, semideviation, etc.\ given some known (fixed) moment information about the associated random variables. Applications of these results include problems in finance (option pricing and portfolio management), inventory control, and supply chain management. Some specific moment problems (that will be investigated in this paper) include the Chebyshev-Cantelli inequality \citep{cantelli1910intorno}, which is a type of classification probability, the Scarf \citep{scarf1958min} and Lo \citep{lo1987semi} bounds, and semideviation of portfolio returns. Some of the original work on bounds for classification probability (that a random variable belongs in a given set) under moment constraints was done by Gauss, Chebyshev \citep{chebyshev1874valeurs}, Markov \citep{markov1884certain}, and Cantelli \citep{cantelli1910intorno}. The term generalized Chebyshev inequalities refers to extensions of these early results. \par
Modern developments in this area involve the use of optimization methods and duality theory to solve particular moment problems in the class above. Pioneering this approach, Isii \citep{isii1960extrema} and Karlin \citep{studden1966tchebycheff}, independently and contemporaneously, showed the sharpness of certain Chebyshev inequalities for univariate random variables. Isii extended this work to the multivariate case \citep{isii1962sharpness}. Smith later replicated this and proposed various applications in decision analysis \citep{smith1995generalized}. Shapiro relaxed the compactness assumptions of Isii's work in his paper on conic linear problems \citep{shapiro2001duality}. More recent research (see next subsection) utilizes the tools of semidefinite programming (SDP) to investigate certain moment problems of interest in a variety of applications (finance, inventory control, supply chain management). \par
The purpose of this work is to extend the notion of a moment problem (MP) to a setting that incorporates ambiguity about the underlying distribution. We coin the term distributionally robust moment problem (DRMP) to reference such problems. These problems are formulated and solved via the framework of Wasserstein discrepancy between distributions and the corresponding (infinite dimensional) problem of moments duality results. The first steps are to define a notion of DRMPs and formulate a primal problem that measures the effect of ambiguity in distribution, as measured by $\delta$. With that in hand, next steps are to formulate and solve the simpler finite dimensional dual problems to quantify the bounds for robustness as a function of ambiguity $\delta$. 
An outline of this paper is as follows. Section 1 gives on overview of moment problems and robustness as well as a literature review. Sections 2 and 3 develop the main theoretical and computational results to characterize tight bounds for DRMPs in a univariate setting. The particular problems of interest are the Chebyshev-Cantelli inequality, the Scarf and Lo bounds, and semideviation. Section 4 applies our results to a few specific problem instances in inventory control and portfolio maangement using the methods of Sections 2 and 3. Section 5 discusses conclusions and suggestions for further research. All detailed proofs are relegated to the Appendix. \par

\subsection{Related Literature}
This subsection reviews some of the more recent research on moment problems using the tools of SDP to investigate certain moment problems in a variety of industry applications. In \citet{bertsimas2002relation}, the authors investigate best possible bounds on option prices given moment information on asset prices. For the univariate problem they show this can be done either in closed form or by using SDP. For the multivariate problem they find valid but not necessarily tight bounds using convex optimization and prove NP-hardness to find the tight bounds. In a subsequent paper, \citet{bertsimas2005optimal} develop an SDP approach to compute tight inequalities for $\Pr(X \in S)$ for a set $S$ defined by polynomial inequalities and a real random vector $X$ characterized by moment constraints up to order $k$. \citet{popescu2005semidefinite} shows how to use SDP to compute optimal bounds on expectations of functions of random variable(s) with given moment constraints and distributional properties such as unimodality, symmetry, convexity, or smoothness. \citet{zuluaga2005conic} use conic programming to study a special class of generalized Chebyshev inequalities. In particular they find optimal bounds on the expected value of piecewise polynomials where the the random variable(s) are subject to given moment constraints. Their focus is on numerical solutions using SDP. \citet{lasserre2008semidefinite} studies SDP relaxations of the generalized problem of moments (GDP) that successively converge to the optimal value. Furthermore, Lasserre explores particular applications in optimization, probability, financial economics, and optimal control. In another work, \citet{he2010bounding} develop primal-dual conic optimization methods to bound $\Pr(X \geq \mu + a)$ where $a$ is a small deviation in relation to the mean $\mu$, given first, second, and fourth moment constraints. They discuss applications to the max-cut problem. In \citet{chen2011tight}, the authors compute bounds for risk measures such as \textit{conditional value-at-risk} (CVaR) or \textit{value-at-risk} (VaR), applicable to robust portfolio selection models. The authors show that single stage problems can be solved analytically. In the particular case of no more than two additional chance constraints, they show the problem can be solved via SDP. In a later work, \citet{blanchet2018distributionally} develop closed form solutions to the distributionally robust Markowitz (mean-variance) portfolio selection problem. In this problem setting the random return vector is subject to a first moment constraint, using the Wasserstein distance metric to measure distributional ambiguity. Finally, we mention the work of \citet{singh2020distributionally} in which the authors develop analytic and SDP methods to characterize distributionally robust profit opportunities in financial markets where portfolio returns are subject to a first moment constraint (lower bound). Once again, the Wasserstein distance metric is used to measure distributional ambiguity. \par

\subsection{Notation and Definitions}
This subsection lays out the notation and definitions used to develop our framework to investigate DRMPs. The approach taken here is to start with the definitions of specific moment problems and introduce a notion of distributional ambiguity via the Wasserstein distance metric. As such, we include definitions for these terms as well as some commentary on the problem of moments duality result used to formulate the dual problems for DRMPs.
The empirical measure, $Q_n$, is defined as
$Q_n := \frac{1}{n} \sum_{i=1}^{n} \delta_{x_i}$
where $\delta_{x_i}$ is a Dirac measure. In our context, the uncertainty set for probability measures is
$U_{\delta}(Q_n) = \{Q: W_d(Q,Q_n) \leq \delta\}$
where $W_d$ is the Wasserstein metric with associated distance function $d(\cdot,\cdot)$ \citep{blanchet2018distributionally}. Define
\begin{equation*}
W_d(Q,Q') = \inf_\pi \{ \mathbb{E}^\pi [d(X,Y)]: X \sim Q, Y \sim Q' \} 
\end{equation*}
where $d(X,Y)$ is the distance between random variables $X$ and $Y$ that follow distributions $Q$ and $Q'$ respectively, and the $\inf$ is taken over all joint distributions $\pi$ with marginals $Q$ and $Q'$. This work uses the (squared) Euclidean distance function $d(x,y) = \| x-y \|^2_2 = \sum_{i=1}^m (x_i - y_i)^2$ \citep {zhao2018data}. \par
The $k^{\text{th}}$ moment of a random variable $X$ is defined as $M_k = \mathbb{E} [ X^k ]$ for $k \in \mathbb{N}_0 := \{0\} \cup \mathbb{N}$. By definition, $M_0 := 1$ and $M_1 = \mu$ denotes the mean of $X$. Certain properties are required to hold for moments of a random variable. For example, Cauchy-Schwarz inequality requires that $M^2_1 \leq M_2$. The $k^{\text{th}}$ \textit{centralized} moment of $X$ is defined as $C_k =  \mathbb{E} ( X  - \mathbb{E} [ X ] )^k$ for $k \in \mathbb{N}_0$. Note that $C_1 = 0$ and $C_2 = \sigma^2 \geq 0$ which denotes the variance of $X$. Furthermore, note that $C_2 = \sigma^2 = M_2 - M^2_1 \implies M_2 = \mu^2 + \sigma^2$. The $k^{\text{th}}$ \textit{lower partial} moment $\text{LPM}_k(\tau)$ of $X$ is defined as $\text{LPM}_k(\tau) =  \mathbb{E} [( \tau - X )^k_+]$ for $k \in \mathbb{N}_0$ \citep[cf e.g.][]{chen2011tight}. Similarly, the $k^{\text{th}}$ \textit{upper partial} moment $\text{UPM}_k(\tau)$ of $X$ is defined as $\text{UPM}_k(\tau) =  \mathbb{E} [( X - \tau )^k_+]$ for $k \in \mathbb{N}_0$. In particular, we have that $\text{LPM}_0(\tau) =  \mathbb{E} [ \mathbbm{1}_{\{X \leq \tau \}} ] = \Pr(X \leq \tau)$, and $\text{UPM}_0(\tau) =  \mathbb{E} [ \mathbbm{1}_{\{X \geq \tau \}} ] = \Pr(X \geq \tau)$. Let $\mathcal{C}_k := \{\text{LPM}_j(\tau): j \in \{0,1,...,k\}\} \cup \{\text{UPM}_j(\tau): j \in \{0,1,...,k\}\}$. Note that $\mathcal{C}_2$ will be the class of moment problems of interest in this research.


\subsubsection{Computing Wasserstein Distance in One Dimension}
This section introduces some standard results on computing Wasserstein distance between one-dimensional distributions. 
The standard results (below) are presented in the online document by \citet{LWonWD}. Wasserstein distance admits simple expressions for univariate distributions. The Wasserstein distance of order $p$ is defined over the set $\mathcal{P}$ of joint distributions $\pi$ with marginals $Q$ and $Q'$ as 
\[
W_p(Q,Q') = \left(  \inf_{\pi \in \mathcal{P}(X,Y)} \int \| x - y \|^p \: d \pi(x,y) \right)^{1/p}.
\]
Note that in this work we consider Wasserstein distance of order $p=2$. When $m=1$ there is the formula 
\[
W_p(Q,Q') = \left( \int_{0}^{1} | F^{-1}(z) - G^{-1}(z) |^p \: dz \right)^{1/p}.
\]
For empirical distributions with $n$ points, there is the formula using order statistics on $(X,Y)$
\[
W_p(Q,Q') = \left( \sum_{i=1}^{n} \| X_{(i)} - Y_{(i)} \|^p \right)^{1/p}.
\]
Additional closed forms are known for: (i) normal distributions, (ii) mappings that relate Wasserstein distance to multi-resolution $L_1$ distance. See \citet{LWonWD} for details. This concludes the brief survey of standard (closed form) results.

For discrete distributions, at least a couple of methods have been recently developed to compute approximate and/or (in the limit) exact Wasserstein distance. For distributions with finite support, and cost matrix $C$, one can compute $W(Q,Q') := \min_{\pi} \langle C, \pi \rangle$ with probability simplex constraints using linear programming (LP) methods of $O(n^3)$ complexity. An entropy regularized version of this, using regularizer $h(\pi) := \sum \pi_{i,j} \log \pi_{i,j}$ gives rise to the Sinkhorn distance
\[
W_{\epsilon}(Q,Q') := \min_{\pi} \langle C, \pi \rangle + \epsilon h(\pi)
\]
which can be solved using iterative Bregman projections via the Sinkhorn algorithm. See \citet{xie2018fast} for details.

\subsubsection{Mapping of Wasserstein Distance to (Statistical) Confidence Level}
A computable, approximate mapping of Wasserstein distance $\delta$ to (statistical) confidence level $\beta = 1 - \alpha$ can be constructed via the relation
\begin{equation}\label{R1}
	\Pr( W_2(Q,Q_n) \geq \delta) \lesssim \exp{\bigg(-n \frac{8r -2 \sqrt{16r^2+16r\delta+24r+12\delta+9}+4\delta+6}{3+4r}\bigg)} = \alpha \tag{R1}
\end{equation}
where $Q_n$ denotes the empirical measure on $n$ data points and $r$ denotes the radius of the support of $X \sim Q_n$. See Section 3.1 of \citet{carlsson2018wasserstein} for further details.

\subsubsection{Univariate Moment Problems}
As mentioned previously, the Chebyshev-Cantelli (zeroth partial moment) inequality and Scarf and Lo bounds on the first partial moment are classical moment problems. 
For random variable $X \in \mathbb{R}$, the lower tail Chebyshev-Cantelli inequality says 
\begin{equation}\label{C1}
\sup_{\{ X \,:\, M_1(X) = \mu,\, M_2(X) = \sigma^2 + \mu^2 \}} \Pr(X \leq \tau) = 
	\begin{cases}
        1, & \text{for } \tau \geq \mu,\\
        \frac{\sigma^2}{\sigma^2 + (\mu - \tau)^2}, & \text{for } \tau < \mu.
     \end{cases} \tag{C1}
\end{equation}
The upper tail Chebyshev-Cantelli inequality is
\begin{equation}\label{C2}
\sup_{\{ X \,:\, M_1(X) = \mu,\, M_2(X) = \sigma^2 + \mu^2 \}} \Pr(X \geq \tau) = 
	\begin{cases}
        1, & \text{for } \tau \leq \mu,\\
        \frac{\sigma^2}{\sigma^2 + (\tau - \mu)^2}, & \text{for } \tau > \mu.
    \end{cases} \tag{C2}
\end{equation}
The lower first partial moment bound says 
\begin{equation}\label{C3}
\sup_{\{ X \,:\, M_1(X) = \mu,\, M_2(X) = \sigma^2 + \mu^2 \}} \mathbb{E} \, [(\tau - X)_+] = 
        \frac{(\tau - \mu) + \sqrt{\sigma^2 + (\mu - \tau)^2}}{2}. \tag{C3}
\end{equation}
The corresponding upper first partial moment bound is
\begin{equation}\label{C4}
\sup_{\{ X \,:\, M_1(X) = \mu,\, M_2(X) = \sigma^2 + \mu^2 \}} \mathbb{E} \, [(X - \tau)_+] = 
        \frac{(\mu - \tau) + \sqrt{\sigma^2 + (\mu - \tau)^2}}{2}. \tag{C4}
\end{equation}
The lower second partial moment bound is given by
\begin{equation}\label{C5}
\sup_{\{ X \,:\, M_1(X) = \mu,\, M_2(X) = \sigma^2 + \mu^2 \}} \mathbb{E} \, [(\tau - X)^2_+] = 
        [(\tau - \mu)_+]^2 + \sigma^2. \tag{C5}
\end{equation}
Similarly, the upper second partial moment bound says 
\begin{equation}\label{C6}
\sup_{\{ X \,:\, M_1(X) = \mu,\, M_2(X) = \sigma^2 + \mu^2 \}} \mathbb{E} \, [(X - \tau)^2_+] = 
        [(\mu - \tau)_+]^2 + \sigma^2 \tag{C6}
\end{equation}
\citep[cf. e.g.][]{chen2011tight, popescu2005semidefinite}.

\subsubsection{Restatement of Problem of Moments Duality}
In Section 2 we formulate the primal and dual problems for the robust upper and lower tail probabilities and partial moments. A key step in our approach is to use duality results to formulate the simpler yet equivalent dual problems. In this context, to enforce the moment constraints for $Q \in \mathcal{U}_{\delta}(Q_n)$, we appeal to the strong duality of linear semi-infinite programs. The dual problem is much more tractable than the primal problem since it only involves the (finite dimensional) data-driven probability measure $Q_n$ as opposed to a continuum of probability measures. This allows us to solve an optimization problem under an empirical measure defined by the chosen data set. A brief restatement of this duality result follows next. See Appendix B of \citet{blanchet2019robust} and Proposition 2 of \citet{blanchet2018distributionally} for further details, or \citet{isii1962sharpness} Theorem 1 for the original result and commentary. \par
\textbf{The problem of moments}. Let $X$ be random vector in probability space $(\Omega, \mathcal{B},\mathcal{P})$ and $(\Omega, \mathcal{B},\mathcal{M}^+)$ where $\mathcal{P}$ and $\mathcal{M}^+$ denote the set of measures and non-negative measures respectively, such that Borel measurable functionals $g, f_1,\dots,f_k$ are integrable. Let $f = (f_1,\dots,f_k) : \Omega \rightarrow \mathbb{R}^k$ be a vector of moment functionals. For a real valued vector $q \in \mathbb{R}^k$, we are interested in the worst case bound
\begin{equation*}
v(q) := \sup \big( \: \mathbb{E}^\pi [g(X)] \: : \: \mathbb{E}^\pi [f(X)] = q \: ; \:\: \pi \in \mathcal{P} \: \big).
\end{equation*}
Adding a constant term by setting $f_0 = \mathbbm{1}_\Omega$, the constraint $\mathbb{E}^\pi[f_0(X)] = 1$, and defining $\tilde{f} = (f_0,f_1,\dots,f_k)$ and $\tilde{q} = (1,q_1,\dots,q_k)$ gives the following reformulation:
\begin{equation*}
v(q) := \sup \big( \: \int g(x) d\pi(x) : \int \tilde{f}(x) d\pi(x) = \tilde{q} \: ; \:\: \pi \in \mathcal{M}^+ \: \big).
\end{equation*}
If a certain Slater-type condition is satisfied, then one has the equivalent dual representation for the above:
\begin{prop*}
Let $\mathcal{Q}_{\tilde{f}} = \{ \int \tilde{f}(x) d\pi(x) : \pi \in \mathcal{M}^+ \}$. If $\tilde{q}$ is an interior point of $\mathcal{Q}_{\tilde{f}}$ then 
\[
v(q) = \inf \big( \: \sum_{i=0}^k a_i q_i : \:\: a_i \in \mathbb{R} ; \:\:\: \sum_{i=0}^k a_i \tilde{f}_i(x) \geq g(x) \:\: \forall x \in \Omega \: \big).
\]
\end{prop*}

The primal problem is concerned with the worst case expected loss for some objective function $g$, under moment constraints.
Note that the primal problem is an infinite dimensional stochastic optimization problem and thus difficult to solve directly. The simplicity and tractability of the dual problem make it quite attractive. \par

\begingroup
\setlength{\parindent}{0pt}
\section{Theory: A Class of Distributionally Robust Risk Measures}
Section 2 develops a semi-analytic method that can solve the dual formulations of the distributionally robust moment problems in polynomial time. The approach involves solving the jointly convex dual problem via use of a one-dimensional line search method in conjunction with a directional descent (DD) method (see algorithm in Section 2.4) that is $\mathcal{O}(n^2)$. Note the DD method requires at most $\mathcal{O}(n^2)$ operations, as it searches at most $n^2$ line segments and regions that partition the half-plane, and is a descent method that only needs to traverse each line segment and/or region once.
Although this approach can compute solutions in polynomial time, it can be challenging to do so in practice, given the complexity of implementing the DD method. A simpler alternative, a computational approach which we refer to as the spherical method (SM), will presented in Section 3.

\subsection{Primal Formulations}
The distributionally robust moment problems are formed by extending the moment constraints to incorporate distributional ambiguity. Specifically, the new constraint set is $\{ X \,:\, M_1(X) = \mu,\, M_2(X) = \sigma^2 + \mu^2,\, Q \in \mathcal{U}_{\delta}(Q_n) \}$. Table 1 below lists the distributionally robust counterpart to each classical moment problem. The naming convention is as follows: (i) L for lower, (ii) U for upper, (iii) Z for zero, (iv) F for first, (v) S for second, (vi) and PM for partial moment. Our generic approach to solving these problems will consist of a few key steps: (i) use problem of moments duality (see Section 1.3.4) to formulation the convex finite dimensional dual problems $\inf_{\{\lambda_1 \geq 0, \lambda_2, \lambda_3\}} F(\lambda_1,\lambda_2,\lambda_3;\psi_\tau)$, (ii) apply an affine change of variables $\xi = \lambda_1 + \lambda_3$ that preserves convexity, (iii) develop a polynomial time algorithm (the DD method) to compute $f(\xi;\psi_\tau) := \min_{\{ \lambda_1 \geq 0, \lambda_2 \}} F(\lambda_1,\lambda_2,\xi;\psi_\tau)$ for fixed $\xi$, (iv) employ a line search method to evaluate $\min_{\{\xi > 0\}} f(\xi;\psi_\tau)$. This approach is applied to solve all six primal problems listed in Table 1. For $\lambda := (\lambda_1,\lambda_2,\lambda_3)$, the functional form for the corresponding dual problems is $\inf_{\{ \lambda_1  \geq 0, \lambda_2, \lambda_3\}} F(\lambda; \psi_\tau) := \lambda_1 \delta + \lambda_2 \mu + \lambda_3 (\sigma^2 + \mu^2) + \frac{1}{n} \sum_{i=1}^n \Psi_i(\lambda;\psi_\tau)$. Function $\Psi_i(\lambda;\psi_\tau)$ will embed a unique cost function $\psi_\tau$ for the corresponding primal problem. 

\renewcommand{\arraystretch}{1.75}
\begin{table}[H]
\normalsize
\begin{center}
\caption{Distributionally Robust Moment Problems}
\begin{tabular}{ |c|c|l| }
 \hline
\textit{Name} & \textit{Primal Problem} & \textit{Tag} \\
 \hline
LZPM & $\sup_{\{ X \,:\, M_1(X) = \mu,\, M_2(X) = \sigma^2 + \mu^2,\, Q \in \mathcal{U}_{\delta}(Q_n) \}} \mathbb{E} \, [\mathbbm{1}_{\{X \leq \tau\}}]$ & (P1) \\ 
\hline
UZPM & $\sup_{\{ X \,:\, M_1(X) = \mu,\, M_2(X) = \sigma^2 + \mu^2,\, Q \in \mathcal{U}_{\delta}(Q_n) \}} \mathbb{E} \, [\mathbbm{1}_{\{X \geq \tau\}}]$ & (P2) \\ 
\hline
LFPM & $\sup_{\{ X \,:\, M_1(X) = \mu,\, M_2(X) = \sigma^2 + \mu^2,\, Q \in \mathcal{U}_{\delta}(Q_n) \}} \mathbb{E} \, [(\tau - X)_+]$  & (P3) \\ 
\hline
UFPM & $\sup_{\{ X \,:\, M_1(X) = \mu,\, M_2(X) = \sigma^2 + \mu^2,\, Q \in \mathcal{U}_{\delta}(Q_n) \}} \mathbb{E} \, [(X - \tau)_+]$ & (P4) \\ 
\hline
LSPM & $\sup_{\{ X \,:\, M_1(X) = \mu,\, M_2(X) = \sigma^2 + \mu^2,\, Q \in \mathcal{U}_{\delta}(Q_n) \}} \mathbb{E} \, [(\tau - X)^2_+]$ & (P5) \\  
\hline
USPM & $\sup_{\{ X \,:\, M_1(X) = \mu,\, M_2(X) = \sigma^2 + \mu^2,\, Q \in \mathcal{U}_{\delta}(Q_n) \}} \mathbb{E} \, [(X - \tau)^2_+]$ & (P6) \\
\hline
\end{tabular} 
\end{center} 
\end{table}
\renewcommand{\arraystretch}{1}

\subsection{Dual Formulations}
As mentioned in the previous subsection, the functional form for the corresponding dual problems is $\inf_{\{ \lambda_1  \geq 0, \lambda_2, \lambda_3\}} F(\lambda;\psi_\tau) := \lambda_1 \delta + \lambda_2 \mu + \lambda_3 (\sigma^2 + \mu^2) + \frac{1}{n} \sum_{i=1}^n \Psi_i(\lambda;\psi_\tau)$. In particular, let us consider the primal problem (P1) for which $\psi_\tau = \mathbbm{1}_{\{x \leq \tau\}}$. Using a duality of moments argument leads to the following result:
\begin{proprep}
The dual problem to (P1) can be written as
\begin{equation}\label{D1}
\inf_{\{ \lambda_1  \geq 0, \lambda_2, \lambda_3\}} F(\lambda;\psi_\tau) := \lambda_1 \delta + \lambda_2 \mu + \lambda_3 (\sigma^2 + \mu^2) + \frac{1}{n} \sum_{i=1}^n \Psi_i(\lambda;\psi_\tau) \tag{D1}
\end{equation}
where $\Psi_i(\lambda;\psi_\tau) := \sup_{\{x \in \mathbb{R}\}} \,\, [ \mathbbm{1}_{\{x \leq \tau\}} - \lambda_1 (x - x_i)^2 - \lambda_2 x - \lambda_3 x^2 ] = - \lambda_1 x_i^2 + \sup_{\{x \in \mathbb{R}\}} \,\, [ \mathbbm{1}_{\{x \leq \tau\}} - (\lambda_1 + \lambda_3) x^2 + (2 \lambda_1 x_i - \lambda_2)x ]$.
\end{proprep}
\begin{proofsketch}
The key result to use here is problem of moments duality. The objective $\Pr(X \leq \tau)$ can be formulated as $\mathbb{E} \, [\mathbbm{1}_{\{X \leq \tau\}}]$. The empirical measure, moment constraints, and Wasserstein distance constraint can be formulated using $f$ as a vector of moment functionals and $q$ as a real valued vector. Since $\tilde{q}$ satisfies the interior point condition, invoking problem of moments duality yields the dual problem (D1) as specified above. See Appendix for the detailed proof.
\end{proofsketch}

\begin{appendixproof}
\begingroup
\setlength{\parindent}{0pt}
We follow the approach in Proposition 2 of \citet{blanchet2018distributionally}. Introduce a slack random variable $S := \alpha$ where $\alpha$ is a scalar value. Then the primal problem (P1) can be reformulated as
\begin{align*}\label{P1b}
\max \, \mathbb{E}^{\pi} [\mathbbm{1}_{\{ U \leq \tau \}}] \, : \{ \, \mathbb{E}^{\pi} [c(U,X)+S] = \alpha,\, \pi_X = Q_n,\, \pi(S = \alpha) = 1,\, \mathbb{E}^{\pi} [U] = \mu,\, \mathbb{E}^{\pi} [U^2] = \sigma^2 + \mu^2\, \} \tag{P1b}
\end{align*}
where $\pi \in \mathcal{P}(\mathbb{R}^m \times \mathbb{R}^m \times \mathbb{R}_+)$. Define $\Omega := \{(u,x,s) : c(u,x) < \infty,\, s \geq 0 \}$, and let 
\[
f(u,x,s) = \begin{bmatrix} \mathbbm{1}_{\{x = x_1\}}(u,x,s) \\ \vdots \\ \mathbbm{1}_{\{x = x_n\}}(u,x,s) \\ u \\ u^2 \\ \mathbbm{1}_{\{s = \alpha\}}(u,x,s) \\ c(u,x) + s \end{bmatrix}
\quad \text{and} \quad
q = \begin{bmatrix} \frac{1}{n} \\ \vdots \\ \frac{1}{n} \\ \mu \\ \sigma^2 + \mu^2 \\ 1 \\ \delta \end{bmatrix}.
\]
Then (P1b) can be reformulated as
\begin{align*}\label{P1c}
\max \, \mathbb{E}^{\pi} [\mathbbm{1}_{\{ U \leq \tau \}}] \, : \{ \, \mathbb{E}^{\pi} [f(U,X,S)] = q,\, \pi \in \mathcal{P}_\Omega\, \}. \tag{P1c}
\end{align*}
Now let $f_0 = \mathbbm{1}_\Omega,\, \tilde{f} = (f_0,f),\, \tilde{q} = (1,q),\, \mathcal{Q}_{\tilde{f}} := \{ \int \tilde{f}(x) d \pi(x) : \pi \in \mathcal{M}^+_\Omega \}$, where $\mathcal{M}^+_\Omega$ denotes the set of non-negative measures on $\Omega$. By inspection it is clear that $\tilde{q}$ lies in the interior of $\mathcal{Q}_{\tilde{f}}$. Therefore, by problem of moments duality, the optimal value of (P1c) is equal to that of its dual problem (D1c) below.
\begin{equation*}\label{D1c}
\inf_{\{a = (a_0,a_1,...,a_{n+3}) \in \mathbb{A}\}} \{ a_0 + \frac{1}{n} \sum_{i=1}^n a_i + a_{n+1} \mu + a_{n+2} (\sigma^2 + \mu^2) + a_{n+3} + a_{n+4} \delta\} \tag{D1c}
\end{equation*}
for $\mathbb{A} := \{ a : a_0 + \frac{1}{n} \sum_{i=1}^n a_i \mathbbm{1}_{\{ x = x_i \}} + a_{n+1} \mu + a_{n+2} (\sigma^2 + \mu^2) + a_{n+3} \mathbbm{1}_{\{ s = \alpha\}} + a_{n+4} [c(u,x)+s] \geq \mathbbm{1}_{\{ u \leq \tau \}},\, \forall \, (u,x,s) \in \Omega \}.$
Replacing $x=x_i$ in the definition of $\mathbb{A}$ gives the set of inequalities 
\[
a_0 + a_i + a_{n+3} \geq \sup_{\{ (u,s) \in \Omega \}} [ \mathbbm{1}_{\{ u \leq \tau \}} - a_{n+1} \mu - a_{n+2} (\sigma^2 + \mu^2) - a_{n+4} [c(u,x_i)+s] ] \, \forall i \, \in {\{1,...,n\}}.
\]
Furthermore, it follows that
\begin{align*}
&\sup_{\{ (u,s) \in \Omega \}} [ \mathbbm{1}_{\{ u \leq \tau \}} - a_{n+1} \mu - a_{n+2} (\sigma^2 + \mu^2) - a_{n+4} [c(u,x_i)+s] ] \\
&= \begin{cases}
	\infty, & \text{if } a_{n+4} < 0, \\ 
	\sup_{\{u \in \mathbb{R}\}} [ \mathbbm{1}_{\{ u \leq \tau \}} - a_{n+1} \mu - a_{n+2} (\sigma^2 + \mu^2) - a_{n+4} c(u,x_i) ], & \text{if } a_{n+4} \geq 0,
	 \end{cases}
\end{align*}
which leads to the following expression for the dual problem
\begin{equation*}\label{D1d}
\inf_{\{a \in \mathbb{R}^{n+4},\, a_{n+4} \geq 0\}} \{ a_0 + \sum_{i=1}^n a_i + a_{n+1} \mu + a_{n+2} (\sigma^2 + \mu^2) + a_{n+3} + a_{n+4} \delta :
 a_0 + a_i + a_{n+3} \geq \Psi_i(a;\psi_\tau) \, \forall i \, \in \{1,...,n\} \} \tag{D1d}
\end{equation*}
where $\Psi_i(a_{n+1},a_{n+2},a_{n+4};\psi_\tau) := \sup_{\{x \in \mathbb{R}\}} [ \mathbbm{1}_{\{ x \leq \tau \}} - a_{n+1} \mu - a_{n+2} (\sigma^2 + \mu^2) - a_{n+4} c(x,x_i) ]$.
Using $\lambda := (\lambda_1,\lambda_2,\lambda_3)$ to replace $(a_{n+4},a_{n+1},a_{n+2})$, the dual problem becomes
\begin{equation*}\label{D1e}
\inf_{\{ \lambda_1 \geq 0, \lambda_2, \lambda_3 \}} \lambda_1 \delta + \lambda_2 \mu + \lambda_3 (\sigma^2 + \mu^2) + \frac{1}{n} \sum_{i=1}^n \Psi_i(\lambda;\psi_\tau) \tag{D1e}
\end{equation*}
where $\Psi_i(\lambda;\psi_\tau) := \sup_{\{x \in \mathbb{R}\}} \,\, [ \mathbbm{1}_{\{x \leq \tau\}} - \lambda_1 (x - x_i)^2 - \lambda_2 x - \lambda_3 x^2 ]$.
\endgroup
\end{appendixproof}

The dual problems for the other primal problems (P2) through (P6) can be formulated in the same way. The proofs are similar to that for (P1) and are omitted.
See Table 2 below for a complete listing.

\renewcommand{\arraystretch}{1.75}
\begin{table}[H]
\normalsize
\begin{center}
\caption{Distributionally Robust Moment Problems}
\begin{tabular}{ |c|c|c||l| }
 \hline
\textit{Name} & $\psi_\tau$ & \textit{Dual Problem} $\Psi_i(\lambda;\psi_\tau)$ function & \textit{Tag} \\
 \hline
LZPM & $\mathbbm{1}_{\{x \leq \tau\}}$ & $\sup_{\{x \in \mathbb{R}\}} \,\, [ \mathbbm{1}_{\{x \leq \tau\}} - \lambda_1 (x - x_i)^2 - \lambda_2 x - \lambda_3 x^2 ]$ & (D1) \\ 
\hline
UZPM & $\mathbbm{1}_{\{x \geq \tau\}}$ & $\sup_{\{x \in \mathbb{R}\}} \,\, [ \mathbbm{1}_{\{x \geq \tau\}} - \lambda_1 (x - x_i)^2 - \lambda_2 x - \lambda_3 x^2 ]$ & (D2) \\ 
\hline
LFPM & $( \tau - x )_+$ & $\sup_{\{x \in \mathbb{R}\}} \,\, [ ( \tau - x )_+ - \lambda_1 (x - x_i)^2 - \lambda_2 x - \lambda_3 x^2 ]$ & (D3) \\ 
\hline
UFPM & $( x - \tau )_+$ & $\sup_{\{x \in \mathbb{R}\}} \,\, [ ( x - \tau )_+ - \lambda_1 (x - x_i)^2 - \lambda_2 x - \lambda_3 x^2 ]$ & (D4) \\ 
\hline
LSPM & $( \tau - x )^2_+$ & $\sup_{\{x \in \mathbb{R}\}} \,\, [  ( \tau - x )^2_+ - \lambda_1 (x - x_i)^2 - \lambda_2 x - \lambda_3 x^2 ]$ & (D5) \\  
\hline
USPM & $( x - \tau )^2_+$ & $\sup_{\{x \in \mathbb{R}\}} \,\, [ ( x - \tau )^2_+ - \lambda_1 (x - x_i)^2 - \lambda_2 x - \lambda_3 x^2 ]$ & (D6) \\
\hline
\end{tabular} 
\end{center} 
\end{table}
\renewcommand{\arraystretch}{1}

\subsection{Lemmas}

Towards solving the dual problems, we use lemmas to evaluate $\Psi_i \; \forall i \in \{1,...,n\}$ for (D1) through (D6). For Lemmas 2.1 - 2.4, let $a > 0$; for Lemmas 2.5 - 2.6, let $a > 1$. Define quadratic $g_0(x; a,b) := -a x^2 + 2b x$ and let the condensed notation $g_0(x)$ suppress the coefficients $\{a,b\}$. Also define $g(a,b;\psi_\tau) := \sup_{\{x \in \mathbb{R}\}} [ \psi_\tau + g_0(x;a,b) ]$. See Appendix for detailed proofs.

\begin{lemmarep}
For $\psi_\tau := \mathbbm{1}_{\{x \leq \tau\}}$,
\[
g(a,b;\psi_\tau) := \sup_{\{x \in \mathbb{R}\}} [ \mathbbm{1}_{\{x \leq \tau\}} + g_0(x;a,b) ]
= \begin{cases}
	1 + g_0(\frac{b}{a}), & \text{if } \tau \geq \frac{b}{a}, \\
	1 + g_0(\tau), & \text{if } \frac{b}{a} - \frac{1}{\sqrt{a}} < \tau < \frac{b}{a}, \\
	g_0(\frac{b}{a}), & \text{if } \tau \leq \frac{b}{a} - \frac{1}{\sqrt{a}}
	\end{cases}
\]
\end{lemmarep}
\begin{proofsketch}
The proof considers two cases: $\tau < x^*$ and $\tau \geq x^*$ where $x^* = \frac{b}{a}$ denotes the critical point for $g_0$. For the latter case, $g$ evaluates to $1 + g_0(x^*)$. For the former case, $g$ evaluates to $\max(1+g_0(\tau),g_0(x^*))$. Working out the bracketing $\tau$ values for these cases (and subcases) leads to the expression for $g$ as given above.
\end{proofsketch}
\begin{appendixproof}
\begingroup
\setlength{\parindent}{0pt}
Let $x^* = \frac{b}{a}$ denote the critical point for $g_0$. The proof proceeds in two cases.
\begin{case}
$x^* > \tau$\\
In this case $g = \sup_{\{x \in \mathbb{R}\}} \, [\mathbbm{1}_{\{x \leq \tau\}} + g_0] = \max( \sup_{\{x \leq \tau \}} \, [1 + g_0],\, \sup_{\{x > \tau\}} \, [g_0] ) = \max(1+g_0(\tau),g_0(x^*))$ since $x^* > \tau$. Thus,
\[
g = \begin{cases}
1 + g_0(\tau), & \text{if } g_0(\tau) > g_0(x^*) - 1, \\
g_0(x^*), & \text{if } g_0(\tau) \leq g_0(x^*) - 1, 
	\end{cases}
\]
where $g_0(\tau) > g_0(x^*) - 1 \implies g_0(x^*) - a(\tau - x^*)^2 > g_0(x^*) - 1$ by Taylor expansion of $g_0(\tau)$ since $x^*$ is the maximum. Simplifying gives $a (\tau - x^*)^2 < 1 \iff x^* - \frac{1}{\sqrt{a}} < \tau < x^*$.
\end{case}
\begin{case}
$x^* \leq \tau$\\
In this case $g = \sup_{\{x \in \mathbb{R}\}} \, [\mathbbm{1}_{\{x \leq \tau\}} + g_0] = \sup_{\{x \leq \tau \}} \, [\mathbbm{1}_{\{x \leq \tau\}} + g_0] = 
\sup_{\{x \leq \tau \}} \, [1 + g_0] = 1 + g_0(x^*)$.
\end{case}
Collecting function values for $g$ and bracketing conditions for $\tau$, we arrive at the expression for $g$ given in the lemma.
\endgroup
\end{appendixproof}

\begin{lemmarep}
For $\psi_\tau := \mathbbm{1}_{\{x \geq \tau\}}$,
\[
g(a,b;\psi_\tau) := \sup_{\{x \in \mathbb{R}\}} [ \mathbbm{1}_{\{x \geq \tau\}} + g_0(x; a,b) ]
= \begin{cases}
	1 + g_0(\frac{b}{a}), & \text{if } \tau \leq \frac{b}{a}, \\
	1 + g_0(\tau), & \text{if } \frac{b}{a} < \tau < \frac{b}{a} + \frac{1}{\sqrt{a}}, \\
	g_0(\frac{b}{a}), & \text{if } \tau \geq \frac{b}{a} + \frac{1}{\sqrt{a}}.
	\end{cases}
\]
\end{lemmarep}
\begin{proofsketch}
The approach is similar to the previous lemma; replace $\mathbbm{1}_{\{x \leq \tau\}}$ with $\mathbbm{1}_{\{x \geq \tau\}}$ in the calculations.
\end{proofsketch}
\begin{appendixproof}
\begingroup
\setlength{\parindent}{0pt}
Let $x^* = \frac{b}{a}$ denote the critical point for $g_0$. The proof proceeds in two cases.
\setcounter{case}{0}
\begin{case}
$x^* < \tau$\\
In this case $g = \sup_{\{x \in \mathbb{R}\}} \, [\mathbbm{1}_{\{x \geq \tau\}} + g_0] = \max( \sup_{\{x \geq \tau \}} \, [1 + g_0],\, \sup_{\{x < \tau\}} \, [g_0] ) = \max(1+g_0(\tau),g_0(x^*))$ since $x^* < \tau$. Thus,
\[
g = \begin{cases}
1 + g_0(\tau), & \text{if } g_0(\tau) > g_0(x^*) - 1, \\
g_0(x^*), & \text{if } g_0(\tau) \leq g_0(x^*) - 1, 
	\end{cases}
\]
where $g_0(\tau) > g_0(x^*) - 1 \implies g_0(x^*) - a(\tau - x^*)^2 > g_0(x^*) - 1$ by Taylor expansion of $g_0(\tau)$ \textit{since} $x^*$ is the maximum. Simplifying gives $a (\tau - x^*)^2 < 1 \iff x^* < \tau < x^* + \frac{1}{\sqrt{a}}$.
\end{case}
\begin{case}
$x^* \geq \tau$\\
In this case $g = \sup_{\{x \in \mathbb{R}\}} \, [\mathbbm{1}_{\{x \geq \tau\}} + g_0] = \sup_{\{x \geq \tau \}} \, [\mathbbm{1}_{\{x \geq \tau\}} + g_0] = 
\sup_{\{x \geq \tau \}} \, [1 + g_0] = 1 + g_0(x^*)$.
\end{case}
Combining the cases gives the result.
\endgroup
\end{appendixproof}

\begin{lemmarep}
For $\psi_\tau := ( \tau - x )_+$,
\[
g(a,b;\psi_\tau) := \sup_{\{x \in \mathbb{R}\}} [ ( \tau - x )_+ + g_0(x;a,b) ]
= \begin{cases}
	\frac{b^2}{a}, & \text{if } \tau \leq \frac{b}{a} - \frac{1}{4a},\\
	\tau + \frac{(b-1/2)^2}{a}, & \text{if } \tau > \frac{b}{a} - \frac{1}{4a}.
	\end{cases}
\]
\end{lemmarep}
\begin{proofsketch}
The proof considers three cases: $\tau < x^*$, $\tau > x^*$, and $\tau = x^*$, where $x^*$ denotes the critical point for $g$. 
For the first case, $g$ evaluates to $\frac{b^2}{a}$. For the second case, $g$ evaluates to $\tau + \frac{(b-1/2)^2}{a}$. Simplifying leads to $g$ as above.
\end{proofsketch}
\begin{appendixproof}
\begingroup
\setlength{\parindent}{0pt}
The first order optimality conditions (for left and right derivatives to bracket zero) say that 
\[
-\mathbbm{1}_{[0,\infty)}(\tau - x) + 2b - 2ax \geq 0 \geq -\mathbbm{1}_{(0,\infty)}(\tau - x) + 2b - 2ax
\]
which leads to three cases for the critical point $x^*$.
\setcounter{case}{0}
\begin{case}
$x^* > \tau$\\
$x^* > \tau \implies x^* = \frac{b}{a}$ \, for \, $\tau < \frac{b}{a}$.
\end{case}
\begin{case}
$x^* < \tau$\\
$x^* < \tau \implies x^* = \frac{b}{a} - \frac{1}{2a}$ \, for \, $\tau > \frac{b}{a} - \frac{1}{2a}$.
\end{case}
\begin{case}
$x^* = \tau$\\
This case violates the optimality condition and hence does not occur.
\end{case}
Consequently, there are three cases for $\tau$.
\setcounter{case}{0}
\begin{case}
$\tau \leq \frac{b}{a} - \frac{1}{2a}$\\
$x^* = \frac{b}{a} \implies g = \frac{b^2}{a}$.
\end{case}
\begin{case}
$\tau \geq \frac{b}{a}$\\
$x^* = \frac{b}{a} - \frac{1}{2a} \implies g = \tau + \frac{(b-1/2)^2}{a}$.
\end{case}
\begin{case}
$\tau \in (\frac{b}{a} - \frac{1}{2a},\frac{b}{a})$\\
$x^* \in \{ \frac{b}{a}, \frac{b}{a} - \frac{1}{2a} \} \implies g = \max{\{ \frac{b^2}{a},\tau + \frac{(b - 1/2)^2}{a} \}} \implies g
= \begin{cases}
	\frac{b^2}{a}, & \text{if } \frac{b}{a} - \frac{1}{2a} < \tau \leq \frac{b}{a} - \frac{1}{4a},\\
	\tau + \frac{(b-1/2)^2}{a}, & \text{if } \frac{b}{a} - \frac{1}{4a} < \tau < \frac{b}{a}.
	\end{cases}
$
\end{case}
As before, combining cases gives the result.
\endgroup
\end{appendixproof}
\begin{lemmarep}
For $\psi_\tau := ( x - \tau )_+$,
\[
g(a,b;\psi_\tau) := \sup_{\{x \in \mathbb{R}\}} [ ( x - \tau )_+ + g_0(x;a,b) ]
= \begin{cases}
	\frac{b^2}{a}, & \text{if } \tau \geq \frac{b}{a} + \frac{1}{4a},\\
	\frac{(b+1/2)^2}{a} - \tau, & \text{if } \tau < \frac{b}{a} + \frac{1}{4a}. 
	\end{cases}
\]
\end{lemmarep}
\begin{proofsketch}
Follow the approach in the previous lemma, exchanging $(\tau - x)_+$ with $(x - \tau)_+$.
\end{proofsketch}
\begin{appendixproof}
\begingroup
\setlength{\parindent}{0pt}
As before, the first order optimality conditions (for left and right derivatives) say that 
\[
\mathbbm{1}_{(0,\infty)}(x - \tau) + 2b - 2ax \geq 0 \geq \mathbbm{1}_{[0,\infty)}(x - \tau) + 2b - 2ax
\]
which leads to three cases for the critical point $x^*$.
\setcounter{case}{0}
\begin{case}
$x^* < \tau$\\
$x^* < \tau \implies x^* = \frac{b}{a}$ \, for \, $\tau > \frac{b}{a}$.
\end{case}
\begin{case}
$x^* > \tau$\\
$x^* > \tau \implies x^* = \frac{b}{a} + \frac{1}{2a}$ \, for \, $\tau < \frac{b}{a} + \frac{1}{2a}$.
\end{case}
\begin{case}
$x^* = \tau$\\
This case violates the optimality condition and hence does not occur.
\end{case}
Consequently, there are three cases for $\tau$.
\setcounter{case}{0}
\begin{case}
$\tau \geq \frac{b}{a} + \frac{1}{2a}$\\
$x^* = \frac{b}{a} \implies g = \frac{b^2}{a}$.
\end{case}
\begin{case}
$\tau \leq \frac{b}{a}$\\
$x^* = \frac{b}{a} + \frac{1}{2a} \implies g_4 = \frac{(b+1/2)^2}{a} - \tau$.
\end{case}
\begin{case}
$\tau \in (\frac{b}{a},\frac{b}{a} + \frac{1}{2a})$\\
$x^* \in \{ \frac{b}{a}, \frac{b}{a} + \frac{1}{2a} \} \implies g = \max{\{\frac{b^2}{a}, \frac{(b+1/2)^2}{a} - \tau \}} \implies g
= \begin{cases}
	\frac{b^2}{a}, & \text{if } \frac{b}{a} + \frac{1}{4a} \leq \tau < \frac{b}{a} + \frac{1}{2a},\\
	\frac{(b+1/2)^2}{a} - \tau, & \text{if } \frac{b}{a} < \tau < \frac{b}{a} + \frac{1}{4a}.  
	\end{cases}
$
\end{case}
Collecting function values for $g$ and bracketing conditions for $\tau$, we arrive at the expression for $g$ given in the lemma.
\endgroup
\end{appendixproof}

\begin{lemmarep}
For $\psi_\tau := ( \tau - x )^2_+$,
\[
g(a,b;\psi_\tau) := \sup_{\{x \in \mathbb{R}\}} [ ( \tau - x )^2_+ + g_0(x;a,b) ]
= \begin{cases}
	\frac{b^2}{a}, & \text{if } \tau \leq \frac{b}{a},\\
	\frac{b^2-2b\tau+a\tau^2}{a-1}, & \text{if } \tau > \frac{b}{a}.
	\end{cases}
\]
\end{lemmarep}
\begin{proofsketch}
Again, we have three cases: $\tau < x^*$, $\tau > x^*$, and $\tau = x^*$. 
For the first case, $g$ evaluates to $\frac{b^2}{a}$. For the second case, $g$ evaluates to $\frac{(b^2-2b\tau+a\tau^2)}{(a-1)}$. This leads to the result for $g$.
\end{proofsketch}
\begin{appendixproof}
\begingroup
\setlength{\parindent}{0pt}
In this case, from the first order optimality conditions, we have that
\[
-2(\tau - x) \cdot \mathbbm{1}_{[0,\infty)}(\tau - x) + 2b - 2ax \geq 0 \geq -2(\tau - x) \cdot \mathbbm{1}_{(0,\infty)}(\tau - x) + 2b - 2ax
\]
which leads to three cases for the critical point $x^*$.
\setcounter{case}{0}
\begin{case}
$x^* > \tau$\\
$x^* > \tau \implies x^* = \frac{b}{a}$ \, for \, $\tau < \frac{b}{a}$.
\end{case}
\begin{case}
$x^* < \tau$\\
$x^* < \tau \implies -2(\tau - x^*) - 2ax^* + 2b = 0 \implies  x^* = \frac{b - \tau}{a - 1}$ \, for \, $\tau > \frac{b - \tau}{a - 1} \implies \tau > \frac{b}{a}$
\, for \, $(a-1) > 0$.
\end{case}
\begin{case}
$x^* = \tau$\\
$x^* = \tau \implies x^* = \frac{b}{a}$ \, for \, $\tau = \frac{b}{a}$.
\end{case}
Consequently, there are two cases for $\tau$.
\setcounter{case}{0}
\begin{case}
$\tau \leq \frac{b}{a}$\\
$x^* = \frac{b}{a} \implies g = \frac{b^2}{a}$.
\end{case}
\begin{case}
$\tau > \frac{b}{a}$\\
$
x^* = \frac{b - \tau}{a - 1} \implies g = (\tau - \frac{b - \tau}{a - 1})^2 - a (\frac{b - \tau}{a - 1})^2 + 2 b (\frac{b - \tau}{a - 1}) = \tau^2 - \frac{2 \tau (b - \tau)}{a-1} - \frac{(b - \tau)^2}{a-1} + \frac{2b(b-\tau)}{a-1}   
 = \frac{b^2-2b\tau+a\tau^2}{a-1}
$ for $(a-1) > 0$, after doing some algebra.
\end{case}
Proceed as before to combine cases to arrive at the result.
\endgroup
\end{appendixproof}

\begin{lemmarep}
For $\psi_\tau := ( x - \tau )^2_+$,
\[
g(a,b;\psi_\tau) := \sup_{\{x \in \mathbb{R}\}} [ ( x - \tau )^2_+ + g_0(x;a,b) ]
= \begin{cases}
	\frac{b^2}{a}, & \text{if } \tau \geq \frac{b}{a},\\
	\frac{b^2-2b\tau+a\tau^2}{a-1}, & \text{if } \tau < \frac{b}{a}.
	\end{cases}
\]
\end{lemmarep}
\begin{proofsketch}
As before, follow the previous lemma, replacing $(\tau - x)^2_+$ with $(x - \tau)^2_+$.
\end{proofsketch}
\begin{appendixproof}
\begingroup
\setlength{\parindent}{0pt}
Once again, applying the first order optimality conditions gives
\[
2(x - \tau) \cdot \mathbbm{1}_{(0,\infty)}(x - \tau) + 2b - 2ax \geq 0 \geq 2(x - \tau) \cdot \mathbbm{1}_{[0,\infty)}(x - \tau ) + 2b - 2ax
\]
which leads to three cases for the critical point $x^*$.
\setcounter{case}{0}
\begin{case}
$x^* < \tau$\\
$x^* < \tau \implies x^* = \frac{b}{a}$ \, for \, $\tau > \frac{b}{a}$.
\end{case}
\begin{case}
$x^* > \tau$\\
$x^* > \tau \implies 2(x^* - \tau) - 2ax^* + 2b = 0 \implies  x^* = \frac{b - \tau}{a - 1}$ \, for \, $\tau < \frac{b - \tau}{a - 1} \implies \tau < \frac{b}{a}$
\, for \, $(a-1) > 0$.
\end{case}
\begin{case}
$x^* = \tau$\\
$x^* = \tau \implies x^* = \frac{b}{a}$ \, for \, $\tau = \frac{b}{a}$.
\end{case}
Consequently, there are two cases for $\tau$.
\setcounter{case}{0}
\begin{case}
$\tau \geq \frac{b}{a}$\\
$x^* = \frac{b}{a} \implies g = \frac{b^2}{a}$.
\end{case}
\begin{case}
$\tau < \frac{b}{a}$\\
$x^* = \frac{\tau - b}{1 - a} \implies g = (\frac{b - \tau}{a - 1} - \tau)^2 - a (\frac{b - \tau}{a - 1})^2 + 2 b (\frac{b - \tau}{a - 1}) = \frac{b^2-2b\tau+a\tau^2}{a-1}
$ for $(a-1) > 0$, using the result of Lemma 2.5, Case 2.
\end{case}
Once again, combining cases gives the result.
\endgroup
\end{appendixproof}

\subsection{Main Results}
The main results of this subsection solve the dual problems (D1) through (D6) and develop a polynomial time algorithm. Recall that the general form for these dual problems is $\inf_{\{ \lambda_1  \geq 0, \lambda_2, \lambda_3\}} F(\lambda;\psi_\tau) := \lambda_1 \delta + \lambda_2 \mu + \lambda_3 (\sigma^2 + \mu^2) + \frac{1}{n} \sum_{i=1}^n \Psi_i(\lambda;\psi_\tau)$. 

\begin{theorem}
The DD method evaluates $f(\xi;\psi_\tau) := \min_{\{ \lambda_1 \geq 0, \lambda_2\}} F(\lambda_1,\lambda_2,\xi;\psi_\tau)$ for $\xi := \lambda_1+\lambda_3 > 0$, in polynomial time.
\end{theorem}
\begin{proof}
 Note the DD method can evaluate $f(\xi;\psi_\tau)$ in at most $\mathcal{O}(n^2)$ operations, as it searches at most $n^2$ line segments and regions that partition the $\{ \lambda_1 \geq 0, \lambda_2 \}$ half-plane, and it is a descent method that only needs to traverse each line segment and/or region once. This once-only traversal property holds due to the joint convexity of $F(\lambda_1,\lambda_2,\xi;\psi_\tau)$.
\end{proof}
The polynomial time algorithm (to compute the solution) uses the DD method to evaluate $f(\xi;\psi_\tau)$ and a one-dimensional line search to minimize convex function $f(\xi;\psi_\tau)$ over $\xi > 0$. 

\begin{algorithm}[!htb]
\renewcommand{\thealgocf}{}
\SetAlgorithmName{DD}{} 
\DontPrintSemicolon
	\KwInput{$\{\xi\:,\: \{x_i\}\:,\: N\:,n\:,\:\delta\:,\:\mu\:,\:\sigma\}$}  
	\KwOutput{$\{y_\xi = f(\xi)\}$}
	Sort $\{ x_i \}$ Decreasing \; 
	Construct lines $\{\lambda_2 = U_i(\lambda_1 \geq 0) \}$ where $U_i ( \lambda_1 ) := 2\lambda_1 x_i - 2\xi \tau$ \;
	Construct lines $\{\lambda_2 = L_i(\lambda_1 \geq 0) \}$ where $L_i( \lambda_1 ) :=  2\lambda_1 x_i - 2(\xi \tau + \sqrt{\xi})$\;
	Compute $\{ V_m \}$, the set of vertices $(\lambda_1,\lambda_2)$ where either $\{U_i \cap L_j \neq \emptyset\}$ or $\lambda_2 \in \{ U_i(\lambda_1=0) \} \cup \{ L_i(\lambda_1=0) \} $ \;
    	Set $k=0$ and the initial search point to be $\lambda_c(k) = V_0$, the vertex with the smallest value for $F$ \;
	\While {$k < N$} 
	{
		Search adjacent regions $\Gamma$ for descent directions $\lambda^\circ_c(k) + t d_\gamma$ where we move towards the min value $\lambda^*_\gamma$ for $F_\Gamma$ \;
		\tcc*[h]{Here $F_\Gamma$ is defined such that $\{\Psi_i\}$ have the same functional form across the entire $(\lambda_1, \lambda_2)$ plane as in region $\Gamma$, where $\Gamma$ is defined by any supporting lines $U_i$ and $L_j. \;\; \lambda^\circ_c(k)$ is an interior point to region $\Gamma$ within $\epsilon$ of  $\lambda_c(k).$ The number of regions $\Gamma$ can vary from 1 to $n+1$.} 
		\\
		\If{ $F(\lambda^*_\gamma) < F(\lambda_c(k))$ }
		{
			\If{ $\lambda^\circ_c(k) + t d_\gamma \: \cap \: \{ U_i \cup L_i \} = \emptyset$ }
		    {
				$\lambda_c(k+1) := \lambda^*_\gamma$ \;
		    }
		    \Else
		    {
				$\{\lambda_j\} :=  \lambda^\circ_c(k) + t d_\gamma \: \cap \: \{ U_i \cup L_i \}$ \;
				$\lambda_c(k+1) := \argmin_{\{\lambda_j\}} \| \lambda_c(k) - \lambda_j \|$ \;
		    }
			$k = k+1$ \;
			\Continue \;

		}
		Search along adjacent rays $R$ (the line segments $\pm{\vec{U_i}}$ and $\pm{\vec{L_j}}$ emanating from point $\lambda_c(k)$) for descent directions $\lambda_c(k) + t d_r$ where we move towards a critical point $\lambda^*_r$ with zero directional derivative for $F$, so $D_{d_r} \, F(\lambda^*_r) = 0$ \;
		\If{ $\{d_r : D_{d_r} \, F(\lambda^*_r) = 0 \} \neq \emptyset$ }
		{
			$\lambda_c(k+1) := \argmin_{\{ \lambda^*_r \}} \| \lambda_c(k) - \lambda^*_r \|$  \;
			$k = k+1$ \;
		}
		\Else
		{
			\tcc*[h]{There are no descent directions via regions or rays so we are at the min value.} 
			\Return $y_\xi = F(\lambda_c(k))$ \;
		}
		
	}
\caption{Directional Descent Method to compute $f(\xi;\psi_\tau)$ for (D1) with $\psi_\tau := \mathbbm{1}_{\{x \leq \tau\}}$}
\end{algorithm}
\begin{remark}
A Matlab implementation of the DD method is available from the corresponding author upon reasonable request.
\end{remark}

\begin{remark}
For the following propositions, let abbreviation \textnormal{cbcipt} denote the phrase ``can be computed in polynomial time".
\end{remark}

\begin{proprep}
The solution to LZPM dual problem (D1) \textnormal{cbcipt} where
\begin{align*}
F(\lambda_1,\lambda_2,\xi;\psi_\tau)  &= \lambda_1 \delta + \lambda_2 \mu + (\xi - \lambda_1) (\sigma^2 + \mu^2) + \frac{1}{n} \sum_{i=1}^n \Psi_i(\lambda_1,\lambda_2,\xi;\psi_\tau), \\
\Psi_i(\lambda_1,\lambda_2,\xi;\psi_\tau) &= -\lambda_1 x_i^2 + 
	\begin{cases} 
	\infty, & \text{if } \xi \leq 0, \\
	1 + \frac{(2\lambda_1 x_i - \lambda_2)^2}{4\xi}, & \text{if } \tau \geq \frac{2\lambda_1 x_i - \lambda_2}{2\xi}, \\
	1 - \xi \tau^2 + (2\lambda_1 x_i - \lambda_2) \tau, & \text{if } \, \frac{2\lambda_1 x_i - \lambda_2}{2\xi} - \frac{1}{\sqrt{\xi}} < \tau < \frac{2\lambda_1 x_i - \lambda_2}{2\xi}, \\
	\frac{(2\lambda_1 x_i - \lambda_2)^2}{4\xi}, & \text{if } \tau \leq \frac{2\lambda_1 x_i - \lambda_2}{2\xi} - \frac{1}{\sqrt{\xi}}.
	\end{cases}
\end{align*}
\end{proprep}
\begin{proofsketch}
The dual problem (D1) is convex in $\lambda$ hence $f(\xi;\psi_\tau)$ is convex. For fixed $\xi$, $f(\xi;\psi_\tau)$ can be evaluated, using the DD method, in at most $\mathcal{O}(n^2)$ operations to find the (global) minimum of a piecewise convex quadratic function in $(\lambda_1,\lambda_2)$. Thus, one can apply a line search method on $f(\xi;\psi_\tau)$. The constraint $\xi > 0$ ensures the piecewise quadratics have finite local minima. Use Lemma 2.1 to do the calculations; see Appendix for a detailed proof.
\end{proofsketch}
\begin{appendixproof}
\begingroup
\setlength{\parindent}{0pt}
Applying Lemma 2.1 to evaluate $\Psi_i(\lambda;\psi_\tau)$ gives
\begin{align*}
\Psi_i(\lambda;\psi_\tau) &= \sup_{\{x \in \mathbb{R}\}} \,\, [ \mathbbm{1}_{\{x \leq \tau\}} - \lambda_1 (x - x_i)^2 - \lambda_2 x - \lambda_3 x^2 ] \\
	&= -\lambda_1 x_i^2 + \sup_{\{x \in \mathbb{R}\}} \,\, [ \mathbbm{1}_{\{x \leq \tau\}} - (\lambda_1 + \lambda_3) x^2 + (2 \lambda_1 x_i - \lambda_2) x ]\\
	&= -\lambda_1 x_i^2 + \sup_{\{x \in \mathbb{R}\}} \,\, [ \mathbbm{1}_{\{x \leq \tau\}} - \xi x^2 + (2 \lambda_1 x_i - \lambda_2) x ] \, (\text{for } \xi := \lambda_1 + \lambda_3) \\
	&= -\lambda_1 x_i^2 + g(\xi,\lambda_1 x_i - \lambda_2/2;\psi_\tau) \, (\text{for } a := \xi,\, b := \lambda_1 x_i - \lambda_2/2) \\
	&= -\lambda_1 x_i^2 + 
	\begin{cases}
	\infty, & \text{if } \xi \leq 0, \\
	1 + g_0(\frac{b}{a}), & \text{if } \tau \geq \frac{b}{a}, \\
	1 + g_0(\tau), & \text{if } \frac{b}{a} - \frac{1}{\sqrt{a}} < \tau < \frac{b}{a}, \\
	g_0(\frac{b}{a}), & \text{if } \tau \leq \frac{b}{a} - \frac{1}{\sqrt{a}}, 
	\end{cases} \\
	&= -\lambda_1 x_i^2 + 
	\begin{cases} 
	\infty, & \text{if } \xi \leq 0, \\
	1 + \frac{(2\lambda_1 x_i - \lambda_2)^2}{4\xi}, & \text{if } \tau \geq \frac{2\lambda_1 x_i - \lambda_2}{4\xi}, \\
	1 - \xi \tau^2 + (2\lambda_1 x_i - \lambda_2) \tau, & \text{if } \frac{2\lambda_1 x_i - \lambda_2}{2\xi} - \frac{1}{\sqrt{\xi}} < \tau < \frac{2\lambda_1 x_i - \lambda_2}{2\xi}, \\
	\frac{(2\lambda_1 x_i - \lambda_2)^2}{4\xi}, & \text{if } \tau \leq \frac{2\lambda_1 x_i - \lambda_2}{2\xi} - \frac{1}{\sqrt{\xi}}.
	\end{cases}
\end{align*}
Substituting this expression for $\Psi_i$ into $F$ gives the desired result. The constraint $\xi > 0$ ensures the piecewise quadratics have finite local minima. Recall that for fixed $\xi$, $f(\xi;\psi_\tau)$ can be evaluated, using the DD method, in at most $\mathcal{O}(n^2)$ operations to find the (global) minimum of a piecewise convex quadratic function in $(\lambda_1,\lambda_2)$.
\endgroup
\end{appendixproof}

\begin{prop}
The solution to UZPM dual problem (D2) \textnormal{cbcipt} where 
\begin{align*}
\Psi_i(\lambda_1,\lambda_2,\xi;\psi_\tau) &= -\lambda_1 x_i^2 + 
	\begin{cases} 
	\infty, & \text{if } \xi \leq 0, \\
	1 + \frac{(2\lambda_1 x_i - \lambda_2)^2}{4\xi}, & \text{if } \tau \leq \frac{2\lambda_1 x_i - \lambda_2}{2\xi}, \\
	1 - \xi \tau^2 + (2\lambda_1 x_i - \lambda_2) \tau, & \text{if } \, \frac{2\lambda_1 x_i - \lambda_2}{2\xi} < \tau < \frac{2\lambda_1 x_i - \lambda_2}{2\xi} + \frac{1}{\sqrt{\xi}}, \\
	\frac{(2\lambda_1 x_i - \lambda_2)^2}{4\xi}, & \text{if } \tau \geq \frac{2\lambda_1 x_i - \lambda_2}{2\xi} + \frac{1}{\sqrt{\xi}}.
	\end{cases}
\end{align*}
\end{prop}
\begin{proof}
Follow the approach for Proposition 2.2, using Lemma 2.2.
\end{proof}

\begin{proprep}
The solution to LFPM dual problem (D3) \textnormal{cbcipt} where
\begin{align*}
\Psi_i(\lambda_1,\lambda_2,\xi;\psi_\tau) &= -\lambda_1 x_i^2 + 
	\begin{cases} 
	\infty, & \text{if } \xi \leq 0, \\
	\frac{(2\lambda_1 x_i - \lambda_2)^2}{4\xi}, & \text{if } \tau \leq \frac{4\lambda_1 x_i - 2 \lambda_2 - 1}{4\xi},\\
	\tau + \frac{(2\lambda_1 x_i - \lambda_2 - 1)^2}{4\xi}, & \text{if } \tau > \frac{4\lambda_1 x_i - 2 \lambda_2 - 1}{4\xi}. 
	\end{cases}
\end{align*}
\end{proprep}
\begin{proofsketch}
Details are similar as before with one exception: note that for fixed $\xi$, $f(\xi;\psi_\tau)$ can be evaluated, using a reduction of the DD method, in at most $\mathcal{O}(n)$ operations. There is no intersection of lines and the $(\lambda_1 \geq 0, \lambda_2)$ half-plane is partitioned into $(n+1)$ adjacent regions. Use Lemma 2.3; see Appendix for a detailed proof.
\end{proofsketch}
\begin{appendixproof}
\begingroup
\setlength{\parindent}{0pt}
Applying lemma 2.5 to evaluate $\Psi_i(\lambda;\psi_\tau)$ gives
\begin{align*}
\Psi_i(\lambda;\psi_\tau) &= \sup_{\{x \in \mathbb{R}\}} \,\, [ (\tau - x)_+ - \lambda_1 (x - x_i)^2 - \lambda_2 x - \lambda_3 x^2 ] \\
	&= -\lambda_1 x_i^2 + \sup_{\{x \in \mathbb{R}\}} \,\, [  (\tau - x)_+ - (\lambda_1 + \lambda_3) x^2 + (2 \lambda_1 x_i - \lambda_2) x ]\\
	&= -\lambda_1 x_i^2 + \sup_{\{x \in \mathbb{R}\}} \,\, [ (\tau - x)_+ - \xi x^2 + (2 \lambda_1 x_i - \lambda_2) x ] \, (\text{for } \xi := \lambda_1 + \lambda_3) \\
	&= -\lambda_1 x_i^2 + g(\xi,\lambda_1 x_i - \lambda_2/2;\psi_\tau) \, (\text{for } a := \xi,\, b := \lambda_1 x_i - \lambda_2/2) \\
	&= -\lambda_1 x_i^2 + 
	\begin{cases}
	\infty, & \text{if } \xi \leq 0, \\
	\frac{b^2}{a}, & \text{if } \tau \leq \frac{b}{a} - \frac{1}{4a},\\
	\tau + \frac{(b-1/2)^2}{a}, & \text{if } \tau > \frac{b}{a} - \frac{1}{4a},
	\end{cases} \\
	&= -\lambda_1 x_i^2 + 
	\begin{cases} 
	\infty, & \text{if } \xi \leq 0, \\
	\frac{(2\lambda_1 x_i - \lambda_2)^2}{4\xi}, & \text{if } \tau \leq \frac{4\lambda_1 x_i - 2 \lambda_2 - 1}{4\xi},\\
	\tau + \frac{(2\lambda_1 x_i - \lambda_2 - 1)^2}{4\xi}, & \text{if } \tau > \frac{4\lambda_1 x_i - 2 \lambda_2 - 1}{4\xi}.
	\end{cases}
\end{align*}
Substituting this expression for $\Psi_i$ into $F$ gives the desired result. The constraint $\xi > 0$ ensures the piecewise quadratics have finite local minima.
\endgroup
\end{appendixproof}


\begin{prop}
The solution to UFPM dual problem (D4) \textnormal{cbcipt} where
\begin{align*}
\Psi_i(\lambda_1,\lambda_2,\xi;\psi_\tau) &= -\lambda_1 x_i^2 + 
	\begin{cases} 
	\infty, & \text{if } \xi \leq 0, \\
	\frac{(2\lambda_1 x_i - \lambda_2)^2}{4\xi}, & \text{if } \tau \geq \frac{4\lambda_1 x_i - 2\lambda_2+1}{4\xi}, \\
	\frac{(2\lambda_1 x_i - \lambda_2 + 1)^2}{4\xi} - \tau, & \text{if } \tau < \frac{4\lambda_1 x_i - 2\lambda_2+1}{4\xi}.
	\end{cases}
\end{align*}
\end{prop}
\begin{proof}
Similar to that for Proposition 2.4; use Lemma 2.4 and the simplified DD method.
\end{proof}

\begin{prop}
The solution to LSPM dual problem (D5) \textnormal{cbcipt} where 
\begin{align*}
\Psi_i(\lambda_1,\lambda_2,\xi;\psi_\tau) &= -\lambda_1 x_i^2 + 
	\begin{cases} 
	\infty, & \text{if } \xi \leq 1, \\
	\frac{(2\lambda_1 x_i - \lambda_2)^2}{4\xi}, & \text{if } \tau \leq \frac{2\lambda_1 x_i - \lambda_2}{2\xi}, \\
	\frac{(\lambda_1 x_i - \lambda_2/2)^2 - (2\lambda_1 x_i - \lambda_2)\tau + \xi \tau^2 }{\xi-1}, & \text{if } \tau > \frac{2\lambda_1 x_i - \lambda_2}{2\xi}.
	\end{cases}
\end{align*}
\end{prop}
\begin{proof}
Details are similar as before, including use of a variation of the DD method, with one exception: the constraint $\xi > 1$ ensures the piecewise quadratics have finite local minima. Use Lemma 2.5.
\end{proof}


\begin{prop}
The solution to USPM dual problem (D6) \textnormal{cbcipt} where 
\begin{align*}
\Psi_i(\lambda_1,\lambda_2,\xi;\psi_\tau) &= -\lambda_1 x_i^2 + 
	\begin{cases} 
	\infty, & \text{if } \xi \leq 1, \\
	\frac{(2\lambda_1 x_i - \lambda_2)^2}{4\xi}, & \text{if } \tau \geq \frac{2\lambda_1 x_i - \lambda_2}{2\xi}, \\
	\frac{(\lambda_1 x_i - \lambda_2/2)^2 - (2\lambda_1 x_i - \lambda_2)\tau + \xi \tau^2 }{\xi-1}, & \text{if } \tau < \frac{2\lambda_1 x_i - \lambda_2}{2\xi}.
	\end{cases}
\end{align*}
\end{prop}
\begin{proof}
Follow the approach for Proposition 2.6; use Lemma 2.6.
\end{proof}
\endgroup

\begingroup
\setlength{\parindent}{0pt}
\section{Spherical Method}
In Section 2, we presented semi-analytic solutions to convex transformations of the dual DRMPs. While these solutions can be computed in polynomial time, with respect to variables $\lambda_1$ and $\lambda_2$, it can be practically challenging to do so. In this section we develop a computational method, the spherical method (SM), which is simpler to implement and reasonably accurate (on our set of test cases in Section 4). The trade-offs between the two methods are that the DD method is more difficult to code but runs faster; SM is simpler to code but runs slower. Note, however, that \textit{parfor} loops (parallel computing) can be used to significantly speed up computational time for SM. The main idea for SM is to change variables to spherical coordinates (for the dual DRMPs) and conduct a grid search on the angles $\theta$ and $\phi$. For the dual formulations $F(\lambda;\psi_\tau)$ in Section 2, let us set
\[
\lambda_1 := r \sin(\theta) \cos(\phi), \quad
\lambda_2 := r \cos(\theta), \quad
\lambda_3 := r \sin(\theta) \sin(\phi)
\]
where $\text{radius } r \geq 0, \; \text{inclination angle } \theta \in [0,\pi], \; \text{and azimuthal angle } \phi \in [0,2\pi)$. The constraint $\lambda_1 \geq 0$ maps to $r \sin(\theta) \cos(\phi) \geq 0$. It turns out that applying spherical transformations to $\lambda$ generates $F(r,\theta,\phi;\psi_\tau)$ which has a simple structure in $r$ such that computing the extremal point $r^*$ given $(\theta,\phi)$ can be done using straightforward methods. The details for each moment problem are worked out in this section.

\subsection{Dual Reformulations}
Let us transform the dual problem (D1), using spherical coordinates, into the dual problem (SD1) given below.
\begin{equation}\label{SD1}
\inf_{\{r \geq 0, \theta \in [0,\pi], \phi \in [0,2\pi)\}} F(r,\theta,\phi;\psi_\tau) := r \sin(\theta) \cos(\phi) \delta + r \cos(\theta) \mu + r \sin(\theta) \sin(\phi) (\sigma^2 + \mu^2) + \frac{1}{n} \sum_{k=1}^n \Psi_k(r,\theta,\phi;\psi_\tau) \tag{SD1}
\end{equation}
where $\Psi_k(r,\theta,\phi;\psi_\tau) := \sup_{\{x \in \mathbb{R}\}} \,\, [ \mathbbm{1}_{\{x \leq \tau\}} - r \sin(\theta) \cos(\phi) (x - x_k)^2 - r \cos(\theta) x - r \sin(\theta) \sin(\phi) x^2 ] = -r \sin(\theta) \cos(\phi) x^2_k + \sup_{\{x \in \mathbb{R}\}}$
$[ \mathbbm{1}_{\{x \leq \tau\}} - r \sin(\theta) (\cos(\phi) + \sin(\phi)) x^2 + (2r\sin(\theta)\cos(\phi)x_k - r\cos(\theta)) x]$ and $(r\sin(\theta) (\cos(\phi) + \sin(\phi)) > 0)$ guarantees a finite value for $\Psi_k$.
The reformulations for the other dual problems (D2) through (D6) can be done in the same way. 
See Table 3 below for a complete listing.

\renewcommand{\arraystretch}{1.75}
\begin{table}[H]
\normalsize
\begin{center}
\caption{Distributionally Robust Moment Problems}
\begin{tabular}{ |c|c|c|l| }
 \hline
\textit{Name} & $\psi_\tau$ & \textit{Reformulated Dual Problem} $\;\Psi_i(\lambda;\psi_\tau)\;$ function & \textit{Tag} \\
 \hline
LZPM & $\mathbbm{1}_{\{x \leq \tau\}}$ & $\sup_{\{x \in \mathbb{R}\}} \,\, [ \mathbbm{1}_{\{x \leq \tau\}} - r \sin(\theta) \cos(\phi) (x - x_k)^2 - r \cos(\theta) x - r \sin(\theta) \sin(\phi) x^2 ]$ & (SD1) \\ 
\hline
UZPM & $\mathbbm{1}_{\{x \geq \tau\}}$ & $\sup_{\{x \in \mathbb{R}\}} \,\, [ \mathbbm{1}_{\{x \geq \tau\}} - r \sin(\theta) \cos(\phi) (x - x_k)^2 - r \cos(\theta) x - r \sin(\theta) \sin(\phi) x^2 ]$ & (SD2) \\ 
\hline
LFPM & $( \tau - x )_+$ & $\sup_{\{x \in \mathbb{R}\}} \,\, [ ( \tau - x )_+ - r \sin(\theta) \cos(\phi) (x - x_k)^2 - r \cos(\theta) x - r \sin(\theta) \sin(\phi) x^2 ]$ & (SD3) \\ 
\hline
UFPM & $( x - \tau )_+$ & $\sup_{\{x \in \mathbb{R}\}} \,\, [ ( x - \tau )_+ - r \sin(\theta) \cos(\phi) (x - x_k)^2 - r \cos(\theta) x - r \sin(\theta) \sin(\phi) x^2 ]$ & (SD4) \\ 
\hline
LSPM & $( \tau - x )^2_+$ & $\sup_{\{x \in \mathbb{R}\}} \,\, [  ( \tau - x )^2_+ - r \sin(\theta) \cos(\phi) (x - x_k)^2 - r \cos(\theta) x - r \sin(\theta) \sin(\phi) x^2 ]$ & (SD5) \\  
\hline
USPM & $( x - \tau )^2_+$ & $\sup_{\{x \in \mathbb{R}\}} \,\, [ ( x - \tau )^2_+ - r \sin(\theta) \cos(\phi) (x - x_k)^2 - r \cos(\theta) x - r \sin(\theta) \sin(\phi) x^2 ]$ & (SD6) \\
\hline
\end{tabular} 
\end{center} 
\end{table}
\renewcommand{\arraystretch}{1}

\subsection{Lemmas}

Towards solving the reformulated dual problems (SD1) through (SD6), we use the following lemmas to evaluate $\Psi_i \; \forall i \in \{1,...,n\}$ for (SD1) through (SD6). For Lemmas 3.1 - 3.4, let $a > 0$; for Lemmas 3.5 - 3.6, let $a > 1$. As before, let quadratic $g_0(x; a,b) := -a x^2 + 2b x$ and let the condensed notation $g_0(x)$ suppress the coefficients $\{a,b\}$. Also define $g(a,b;\psi_\tau) := \sup_{\{x \in \mathbb{R}\}} [ \psi_\tau + g_0(x;a,b) ]$.

\begin{lemma}
For $\psi_\tau := \mathbbm{1}_{\{x \leq \tau\}}$,
\[
g(a,b;\psi_\tau) := \sup_{\{x \in \mathbb{R}\}} [ \mathbbm{1}_{\{x \leq \tau\}} + g_0(x;a,b) ]
= \max( 1 + g_0(\tau), \mathbbm{1}_{\{(b/a) \leq \tau\}} + g_0(b/a)).
\]
\end{lemma}
\begin{proof}
This characterization follows from inspection of the proof of Lemma 2.1 which considers two cases: $x^* > \tau$ and $x^* \leq \tau$ where $x^* = \frac{b}{a}$ denotes the critical point for $g_0$. For the latter case, $g$ evaluates to $1 + g_0(x^*) = \mathbbm{1}_{\{(b/a) \leq \tau\}} + g_0(b/a) \geq 1 + g_0(\tau)$. For the former case, $g$ evaluates to $\max(1+g_0(\tau),g_0(x^*))$ where $g_0(x^*) = \mathbbm{1}_{\{(b/a) \leq \tau\}} + g_0(b/a)$. Taking the max over both cases gives the expression in Lemma 3.1 above.
\end{proof}

\begin{lemma}
For $\psi_\tau := \mathbbm{1}_{\{x \geq \tau\}}$,
\[
g(a,b;\psi_\tau) := \sup_{\{x \in \mathbb{R}\}} [ \mathbbm{1}_{\{x \geq \tau\}} + g_0(x;a,b) ]
= \max( 1 + g_0(\tau), \mathbbm{1}_{\{(b/a) \geq \tau\}} + g_0(b/a)).
\]
\end{lemma}
\begin{proof}
The approach is similar to the previous lemma; replace $\mathbbm{1}_{\{x \leq \tau\}}$ with $\mathbbm{1}_{\{x \geq \tau\}}$ in the calculations.
\end{proof}

\begin{lemma}
For $\psi_\tau := ( \tau - x )_+$,
\[
g(a,b;\psi_\tau) := \sup_{\{x \in \mathbb{R}\}} [ ( \tau - x )_+ + g_0(x;a,b) ]
= \frac{b^2}{a} + \bigg(\tau - \bigg(\frac{b}{a} - \frac{1}{4a}\bigg) \bigg)_+
\]
\end{lemma}
\begin{proof}
The expression above equals the result for Lemma 2.3 in both cases. Observe $(\tau - (\frac{b}{a} - \frac{1}{4a}) )_+$ is zero for $\tau \leq \frac{b}{a} - \frac{1}{4a}$. And for $\tau > (\frac{b}{a} - \frac{1}{4a})$, $\frac{b^2}{a} + (\tau - (\frac{b}{a} - \frac{1}{4a}) )_+ = \tau + \frac{(b - 1/2)^2}{a}$.
\end{proof}

\begin{lemma}
For $\psi_\tau := ( x - \tau )_+$,
\[
g(a,b;\psi_\tau) := \sup_{\{x \in \mathbb{R}\}} [ ( x - \tau )_+ + g_0(x;a,b) ]
= \frac{b^2}{a} + \bigg( \bigg(\frac{b}{a} + \frac{1}{4a}\bigg)  - \tau \bigg)_+
\]
\end{lemma}
\begin{proof}
This is equivalent to Lemma 2.4. Note that $( (\frac{b}{a} + \frac{1}{4a}) - \tau)_+$ is zero for $\tau \geq \frac{b}{a} + \frac{1}{4a}$. And for $\tau < (\frac{b}{a} + \frac{1}{4a})$, $\frac{b^2}{a} + ( (\frac{b}{a} + \frac{1}{4a})  - \tau)_+ = \frac{(b + 1/2)^2}{a} - \tau$.
\end{proof}

\begin{lemma}
For $\psi_\tau := ( \tau - x )^2_+$,
\[
g(a,b;\psi_\tau) := \sup_{\{x \in \mathbb{R}\}} [ ( \tau - x )^2_+ + g_0(x;a,b) ]
= \frac{b^2}{a} + \frac{a}{a-1} \bigg(\tau - \frac{b}{a} \bigg)^2_+
\]
\end{lemma}
\begin{proof}
This formula agrees with Lemma 2.5. Observe $(\tau - \frac{b}{a})^2_+$ is zero for $\tau \leq \frac{b}{a}$. And for $\tau > \frac{b}{a}$, $\frac{b^2}{a} + \frac{a}{a-1} (\tau - \frac{b}{a})^2_+ =  \frac{b^2}{a} + \frac{a\tau^2-2b\tau+(b^2/a)}{a-1} = \frac{b^2 - 2b\tau + a\tau^2}{a-1}$ using a partial fractions decomposition.
\end{proof}

\begin{lemma}
For $\psi_\tau := ( x - \tau )^2_+$,
\[
g(a,b;\psi_\tau) := \sup_{\{x \in \mathbb{R}\}} [ ( x - \tau )^2_+ + g_0(x;a,b) ]
= \frac{b^2}{a} + \frac{a}{a-1} \bigg(\frac{b}{a} - \tau \bigg)^2_+
\]
\end{lemma}
\begin{proof}
Compare vs.\ Lemma 2.6. Note that $(\frac{b}{a} - \tau)^2_+$ is zero for $\tau \geq \frac{b}{a}$. And for $\tau < \frac{b}{a}$, $\frac{b^2}{a} + \frac{a}{a-1} (\frac{b}{a} - \tau)^2_+ =  \frac{b^2}{a} + \frac{a\tau^2-2b\tau+(b^2/a)}{a-1} = \frac{(b^2 - 2b\tau + a\tau^2)}{a-1}$.
\end{proof}

\subsection{Main Results}
The main results compute approximate numerical solutions to the reformulated dual problems (SD1) through (SD6). Recall that the general form is $\inf_{\{r \geq 0, \theta \in [0,\pi], \phi \in [0,2\pi)\}} F(r,\theta,\phi;\psi_\tau) := r \sin(\theta) \cos(\phi) \delta + r \cos(\theta) \mu + r \sin(\theta) \sin(\phi) (\sigma^2 + \mu^2) + \frac{1}{n} \sum_{k=1}^n \Psi_k(r,\theta,\phi;\psi_\tau)$. SM, described below, is used to evaluate $\inf_{\{r \geq 0\}} F(r,\theta_i,\phi_j)$ for $(\theta_i,\phi_j)$ given. Let us begin with (SD1). 

\begin{algorithm}[!htb]
\renewcommand{\thealgocf}{}
\SetAlgorithmName{SM}{} 
\DontPrintSemicolon
	\KwInput{$\{\theta_i \:,\: \phi_j \:,\: \{x_k\}\:,\: N\:,\: \delta\}$ }
  	\KwOutput{ $\{y_{(\theta_i,\phi_j)} = \inf_{\{r \geq 0\}} F(r,\theta_i,\phi_j)\}$ }
	Select suitable $\epsilon$ such that $\epsilon \gtrsim 0$ \;
	Compute $\{ r_k \}$ such that $1 + g_0(\tau; a_k,b_k) = \mathbbm{1}_{\{(b_k/a_k) \leq \tau\}} + g_0(\frac{b_k}{a_k};a_k,b_k),$ where $r_k \geq \epsilon$ \;
	\If{ $\{r_k \} = \emptyset$ }
    {
		return $y_{(\theta_i,\phi_j)} = F(\epsilon,\theta_i,\phi_j)$ \;
    }
    \Else
    {
		Sort $\{ r_k \}$ Increasing \;
    		$k = 1$ \;
		\While{ $k \leq \vert \{ r_k \} \vert$  }
		{
			\If{ $\partial^-_r F(r_{k}) \leq 0 \leq \partial^+_r F(r_{k})$ }
			{
				$k^* = k$ \;
				return $y_{(\theta_i,\phi_j)} = F(r_{k^*},\theta_i,\phi_j)$ \;
			}
			\Else
			{
			$k = k + 1$ \;
			}
		}
    }
\caption{Spherical Method  to compute $y_{(\theta_i,\phi_j)} = \inf_{\{r \geq 0\}} F(r,\theta_i,\phi_j;\psi_\tau)$ for (SD1) with $\psi_\tau := \mathbbm{1}_{\{x \leq \tau\}}$}
\end{algorithm}

\begin{proprep}
The solution to LZPM dual problem (SD1) can be computed (approximately) using a two-dimensional grid search in angles $(\theta,\phi)$ subject to the constraint $\sin(\theta) (\cos(\phi) + \sin(\phi)) > 0$ and evaluating function $F$ for each point in a list $L$ of tuples $(\theta_i,\phi_j,r_{k^*})$. The list $L$ can be constructed by applying SM to do a linear search on at most $n$ breakpoints $r_k$ to find the extremal point $r_{k^*}$ such that either $0 \in \partial F(r_{k^*})$ or ${r_{k^*}} \gtrsim 0$ given $(\theta_i,\phi_j)$,
where points $r_k$ either satisfy the following relation for $a_k := r_k \sin(\theta_i) (\cos(\phi_j) + \sin(\phi_j)),\; b_k := r_k \sin(\theta_i) \cos(\phi_j) x_k - (\frac{r_k}{2}) \cos(\theta_i),$
\[
1 + g_0(\tau; a_k,b_k) = \mathbbm{1}_{\{(b_k/a_k) \leq \tau\}} + g_0(\frac{b_k}{a_k};a_k,b_k)
\]
or default to $\epsilon \gtrsim 0$ otherwise. The optimality condition for subgradient $\partial F(r_{k^*})$ says that
\[
\partial^-_r F(r_{k^*}) \leq 0 \leq \partial^+_r F(r_{k^*})
\]
where the left and right partial derivatives of $F$ evaluated at $r = r_{k^*}$ for $k^* \in \{1,...,n\}$ are given by
\begin{align*}
\partial^{-}_r F(r_{k^*}) &= \alpha  - \frac{\alpha_0}{n} \sum_k x^2_k  + \frac{1}{n} \bigg[ \; \sum_{k} g_0(\tau;\tilde{a}_k,\tilde{b}_k) + \sum_{k \in K_1} g_0(\frac{\tilde{b}_k}{\tilde{a}_k};\tilde{a}_k,\tilde{b}_k) - g_0(\tau;\tilde{a}_k,\tilde{b}_k) \; \bigg], \\
\partial^{+}_r F(r_{k^*}) &= \alpha - \frac{\alpha_0}{n} \sum_k x^2_k  + \frac{1}{n} \bigg[ \; \sum_{k} g_0(\tau;\tilde{a}_k,\tilde{b}_k) + \sum_{k \in K_2} g_0(\frac{\tilde{b}_k}{\tilde{a}_k};\tilde{a}_k,\tilde{b}_k) - g_0(\tau;\tilde{a}_k,\tilde{b}_k) \; \bigg],
\end{align*}
where \; $\alpha_0 := \sin(\theta_i) \cos(\phi_j) $,\; $\alpha := \alpha_0 \delta + \cos(\theta_i) \mu + \sin(\theta_i) \sin(\phi_j) (\sigma^2 + \mu^2)$,\; $\tilde{a}_k := \sin(\theta_i) (\cos(\phi_j) + \sin(\phi_j)),\; \\ 
\tilde{b}_k := \sin(\theta_i) \cos(\phi_j) x_k - (\frac{1}{2}) \cos(\theta_i),\,$ 
\begin{align*}
K_1 &:= \{ k: 1 + g_0(\tau;a_{k^*},b_{k,k^*}) < \mathbbm{1}_{\{(b_{k,k^*}/a_{k^*}) \leq \tau\}} + g_0(\frac{b_{k,k^*}}{a_{k^*}};a_{k^*},b_{k,k^*}) \},\\
K_2 &:= \{ k: 1 + g_0(\tau;a_{k^*},b_{k,k^*}) \leq \mathbbm{1}_{\{(b_{k,k^*}/a_{k^*}) \leq \tau\}} + g_0(\frac{b_{k,k^*}}{a_{k^*}};a_{k^*},b_{k,k^*}) \},
\end{align*}
for $b_{k,k^*} := r_{k^*} \sin(\theta_i) \cos(\phi_j) x_k - (\frac{r_{k^*}}{2}) \cos(\theta_i)$. 
\end{proprep}
\begin{proofsketch}
Compute the (at most) $n$ breakpoints $r_k$ such that $1 + g_0(\tau; a_k,b_k) = \mathbbm{1}_{\{(b_k/a_k) \leq \tau\}} + g_0(\frac{b_k}{a_k};a_k,b_k)$. Use Lemma 3.1 to deduce that the critical value $x^*_k = b_k/a_k$ for $\Psi_k(r_k,\theta_i,\phi_j;\psi_\tau) = -r_k \sin(\theta_i) \cos(\phi_j) x^2_k + g(a_k,b_k;\psi_\tau)$ does not depend on $r_k$. 
Furthermore, for fixed $k^*$, $\partial^-_r F(r_{k^*})$ and $\partial^+_r F(r_{k^*})$ only depend on $r_{k^*}$ through the index set relations $k \in K_{\{1,2\}}$.
 As the functions $\Psi_k$ and hence $F$ are convex in $r$ for $(\theta_i,\phi_j)$ fixed, it follows that one of these breakpoints $r_k$ must be the extremal point $r_{k^*}$ such that either $0 \in \partial F(r_{k^*})$ or ${r_{k^*}} \gtrsim 0$. Note this set is non-empty since the dual problem (SD1) has finite value. As we are doing a grid search in angles $(\theta,\phi)$, concatentate a list $L$ of tuples $(\theta_i,\phi_j,r_{k^*})$ under the constraint $\sin(\theta_i) (\cos(\phi_j) + \sin(\phi_j)) > 0$, evaluate $F(r_{k^*},\theta_i,\phi_j;\psi_\tau)$ for each point in the list $L$, and return the $\min$ of all these as the (approximate) solution to the dual problem (SD1). See Appendix for details.
\end{proofsketch}
\begin{appendixproof}
\begingroup
\setlength{\parindent}{0pt}
Suppose $(\theta_i,\phi_j)$ are fixed and $a_k := r_k \sin(\theta_i) (\cos(\phi_j) + \sin(\phi_j)),\; b_k := r_k \sin(\theta_i) \cos(\phi_j) x_k - (\frac{r_k}{2}) \cos(\theta_i),$ where points $r_k$ for $k \in \{1,...,n\}$ satisfy the relation 
\[
1 + g_0(\tau; a_k,b_k) = \mathbbm{1}_{\{(b_k/a_k) \leq \tau\}} + g_0(\frac{b_k}{a_k};a_k,b_k)
\]
as in the statement of the theorem. Applying Lemma 3.1 to $\Psi_k(r_k,\theta_i,\phi_j;\psi_\tau) = -r_k \sin(\theta_i) \cos(\phi_j) x^2_k + g(a_k,b_k;\psi_\tau)$, we deduce that the critical value $x^*_k = b_k/a_k$ does not depend on $r_k$. For $k^* \in \{1,...,n\}$, a straightforward calculation shows that the left and right partial derivatives of $F$ evaluated at $r = r_{k^*}$ are given by
\begin{align*}
\partial^{-}_r F(r_{k^*}) &= \alpha - \frac{\alpha_0}{n} \sum_k x^2_k  + \frac{1}{n} \bigg[ \; \sum_{k} g_0(\tau;\tilde{a}_k,\tilde{b}_k) + \sum_{k \in K_1} g_0(\frac{\tilde{b}_k}{\tilde{a}_k};\tilde{a}_k,\tilde{b}_k) - g_0(\tau;\tilde{a}_k,\tilde{b}_k) \; \bigg], \\
\partial^{+}_r F(r_{k^*}) &= \alpha - \frac{\alpha_0}{n} \sum_k x^2_k + \frac{1}{n} \bigg[ \; \sum_{k} g_0(\tau;\tilde{a}_k,\tilde{b}_k) + \sum_{k \in K_2} g_0(\frac{\tilde{b}_k}{\tilde{a}_k};\tilde{a}_k,\tilde{b}_k) - g_0(\tau;\tilde{a}_k,\tilde{b}_k)\; \bigg],
\end{align*}
where \; $\alpha_0 := \sin(\theta_i) \cos(\phi_j) $,\; $\alpha := \alpha_0 \delta + \cos(\theta_i) \mu + \sin(\theta_i) \sin(\phi_j) (\sigma^2 + \mu^2)$,\; $\tilde{a}_k := \sin(\theta_i) (\cos(\phi_j) + \sin(\phi_j)),\; \\
\tilde{b}_k := \sin(\theta_i) \cos(\phi_j) x_k - (\frac{1}{2}) \cos(\theta_i)\,$ 
\begin{align*}
K_1 &:= \{ k: 1 + g_0(\tau;a_{k^*},b_{k,k^*}) < \mathbbm{1}_{\{(b_{k,k^*}/a_{k^*}) \leq \tau\}} + g_0(\frac{b_{k,k^*}}{a_{k^*}};a_{k^*},b_{k,k^*}) \},\\
K_2 &:= \{ k: 1 + g_0(\tau;a_{k^*},b_{k,k^*}) \leq \mathbbm{1}_{\{(b_{k,k^*}/a_{k^*}) \leq \tau\}} + g_0(\frac{b_{k,k^*}}{a_{k^*}};a_{k^*},b_{k,k^*}) \},
\end{align*}
for $a_{k^*} := r_{k^*} \sin(\theta_i) (\cos(\phi_j) + \sin(\phi_j)),\; b_{k,k^*} := r_{k^*} \sin(\theta_i) \cos(\phi_j) x_k - (\frac{r_{k^*}}{2}) \cos(\theta_i)$. 
Furthermore, $\partial^-_r F(r_{k^*})$ and $\partial^+_r F(r_{k^*})$ only depend on $r_{k^*}$ through the index set relations $k \in K_{\{1,2\}}$. As the functions $\Psi_k$ are a pointwise max of a family of affine functions in $r$, we conclude that $\Psi_k$ and hence $F$ are convex in $r$ for $(\theta_i,\phi_j)$ fixed. It follows that one of these breakpoints $r_k$ must be the extremal point $r_{k^*}$ such that either $\partial^-_r F(r_{k^*}) \leq 0 \leq \partial^+_r F(r_{k^*})$ and hence $0 \in \partial F(r_{k^*})$ or ${r_{k^*}} \gtrsim 0$. Note the set $\{ r > 0 : \partial F(r) \geq 0 \}$ is non-empty since the dual problem (SD1) has finite value. As we are doing a grid search in angles $(\theta,\phi)$, concatentate a list $L$ of tuples $(\theta_i,\phi_j,r_{k^*})$ under the constraint $\sin(\theta_i) (\cos(\phi_j) + \sin(\phi_j)) > 0$, evaluate $F(r_{k^*},\theta_i,\phi_j;\psi_\tau)$ for each point in the list $L$, and return the $\min$ of all these as the (approximate) solution to the dual problem (SD1). 

\endgroup
\end{appendixproof}


\begin{prop}
The solution to UZPM dual problem (SD2) can be computed (approximately) by replacing $\mathbbm{1}_{\{x \leq \tau\}}$ with 
$\mathbbm{1}_{\{x \geq \tau\}}$, and following the approach described in Proposition 3.1.
\end{prop}
\begin{proof}
Replace $\mathbbm{1}_{\{x \leq \tau\}}$ with $\mathbbm{1}_{\{x \geq \tau\}}$, apply Lemma 3.2, and use the above approach.
\end{proof}

%

\begin{prop}
The solution to LFPM dual problem (SD3) can be computed (approximately) using a variation of the previous approach. 
For the variation, the list $L$ is now constructed by doing a linear search on $r \geq 0\,$ for each tuple $(\theta_i,\phi_j)$ to find the critical point $r^*_{ij}$ 
where the optimality condition for subgradient $\partial F(r^*_{ij})$ says that
\[
\partial^-_r F(r^*_{ij}) \leq 0 \leq \partial^+_r F(r^*_{ij})
\]
where the left and right partial derivatives of $F$ evaluated at $r = r^*_{ij}$ are
\begin{align*}
\partial^{-}_r F(r^*_{ij}) &= \alpha - \frac{\alpha_0}{n} \sum_k x^2_k + \frac{1}{n} \bigg[ \; \sum_{k} \beta_k - \sum_{k \in K_1} \frac{1}{4 (r^*_{ij})^2 \tilde{a}} \; \bigg], \\
\partial^{+}_r F(r^*_{ij}) &= \alpha - \frac{\alpha_0}{n} \sum_k x^2_k + \frac{1}{n} \bigg[ \; \sum_{k} \beta_k - \sum_{k \in K_2} \frac{1}{4 (r^*_{ij})^2 \tilde{a}} \; \bigg],
\end{align*}
where \; $\alpha_0 := \sin(\theta_i) \cos(\phi_j) $,\; $\alpha := \alpha_0 \delta + \cos(\theta_i) \mu + \sin(\theta_i) \sin(\phi_j) (\sigma^2 + \mu^2)$,\; $\tilde{a} := \sin(\theta_i) (\cos(\phi_j) + \sin(\phi_j)),\; \\
\tilde{b}_k := \sin(\theta_i) \cos(\phi_j) x_k - (\frac{1}{2}) \cos(\theta_i)\,$, $\beta_k =  (\tilde{b}^2_k / \tilde{a})$, 
\begin{align*}
K_1(r^*_{ij}) &:= \left\{ k: \bigg(\tau - \bigg(\frac{\tilde{b}_k}{\tilde{a}} - \frac{1}{4 r^*_{ij} \tilde{a}}\bigg)\bigg) \geq 0 \right\},\\
K_2(r^*_{ij}) &:= \left\{ k: \bigg(\tau - \bigg(\frac{\tilde{b}_k}{\tilde{a}} - \frac{1}{4 r^*_{ij} \tilde{a}}\bigg)\bigg) > 0 \right\}.
\end{align*}
\end{prop}
\begin{proof}
Use Lemma 3.3 to deduce that the breakpoint for the value of $\Psi_k(r_{ij},\theta_i,\phi_j;\psi_\tau) = -r_{ij} \sin(\theta_i) \cos(\phi_j) x^2_k + g(a,b_k;\psi_\tau)$ for $a := r_{ij} \sin(\theta_i) (\cos(\phi_j) + \sin(\phi_j)),\; b_k := r_{ij} \sin(\theta_i) \cos(\phi_j) x_k - (\frac{r_{ij}}{2}) \cos(\theta_i),$ occurs when $(\tau - (\frac{\tilde{b}_k}{\tilde{a}} - \frac{1}{4 r_{ij} \tilde{a}})) = 0$. 
Note the functions $\Psi_k$ and hence $F$ are convex in $r$ for $(\theta_i,\phi_j)$ fixed. Also, under the constraint $r \sin(\theta) (\cos(\phi) + \sin(\phi)) > 0$, recall the dual problem (SD3) has finite value which implies $F$ has finite value below and hence a non-negative subgradient $\partial F$ for some $r_{ij}$ with $(\theta_i,\phi_j)$ fixed. Note that $\partial^-_r F(r_{ij})$ and $\partial^+_r F(r_{ij})$ only depend on $r_{ij}$ through the $K_{\{1,2\}}$ summation terms. Therefore, the asymptotic properties of $F$ are such that its subgradient $\partial F$ crosses zero as $r_{ij}$ sweeps from $0$ to $\infty$ (see \citet{singh2019distributionally} for a similar but more detailed argument). It follows that there exists a critical point $r^*_{ij} > 0$ such that $\partial^-_r F(r^*_{ij}) \leq 0 \leq \partial^+_r F(r^*_{ij})$ and hence $0 \in \partial F(r^*_{ij})$. 
As before, concatentate a list $L$ of tuples $(\theta_i,\phi_j,r^*_{ij})$ under the constraint $\sin(\theta_i) (\cos(\phi_j) + \sin(\phi_j)) > 0$, evaluate $F(r^*_{ij},\theta_i,\phi_j;\psi_\tau)$, and return the $\min$ as the (approximate) solution to dual problem (SD3). 
\end{proof}

%

\begin{prop}
The solution to UFPM dual problem (SD4) can be computed (approximately) by redefining the sets $K_{\{1,2\}}$, and using the approach described in Proposition 3.3.
In particular, the sets $K_{\{1,2\}}$ are now defined as
\begin{align*}
K_1(r^*_{ij}) &:= \left\{ k: \bigg( \bigg(\frac{\tilde{b}_k}{\tilde{a}} + \frac{1}{4 r^*_{ij} \tilde{a}}\bigg) - \tau \bigg) \geq 0 \right\},\\
K_2(r^*_{ij}) &:= \left\{ k: \bigg( \bigg(\frac{\tilde{b}_k}{\tilde{a}} + \frac{1}{4 r^*_{ij} \tilde{a}}\bigg) - \tau \bigg) > 0 \right\}.
\end{align*}
\end{prop}
\begin{proof}
Use the new definition for the sets $K_{\{1,2\}}$ and apply Lemma 3.4 instead of 3.3. Otherwise, the details are similar to those in the proof of Proposition 3.3 and are omitted.
\end{proof}

%

\begin{prop}
The solution to dual problem (SD5) can be computed (approximately) using a variation of the approach in Proposition 3.3.
For the variation, the list $L$ is now constructed by doing a linear search on $\{r : r \sin(\theta_i) (\cos(\phi_j) + \sin(\phi_j)) > 1\}$ for each tuple $(\theta_i,\phi_j)$ to find the extremal point $r^{*}_{ij}$
such that either $\partial_r F(r^*_{ij}) = 0$ or $r^*_{ij} \gtrsim 1/\sin(\theta_i) (\cos(\phi_j) + \sin(\phi_j))$ given $(\theta_i,\phi_j)$.
Furthermore, the continuous partial derivative of $F$ evaluated at $r = r^*_{ij}$ is
\begin{align*}
\partial_r F(r^*_{ij}) &= \alpha - \frac{\alpha_0}{n} \sum_k x^2_k + \frac{1}{n} \bigg[ \; \sum_{k} \beta_k - \sum_{k \in K_1} \frac{\tilde{a}C_k}{(r^*_{ij} \tilde{a} - 1)^2}  \; \bigg], 
\end{align*}
where \; $\alpha_0 := \sin(\theta_i) \cos(\phi_j) $,\; $\alpha := \alpha_0 \delta + \cos(\theta_i) \mu + \sin(\theta_i) \sin(\phi_j) (\sigma^2 + \mu^2)$,\; $\tilde{a} := \sin(\theta_i) (\cos(\phi_j) + \sin(\phi_j)),\; \\
\tilde{b}_k := \sin(\theta_i) \cos(\phi_j) x_k - (\frac{1}{2}) \cos(\theta_i)\,$, $\beta_k =  (\tilde{b}^2_k / \tilde{a}),\, C_k = (\tau - \tilde{b}_k / \tilde{a})^2$, 
\begin{align*}
K_1 &:= \left\{ k: \bigg(\tau - \frac{\tilde{b}_k}{\tilde{a}} \bigg) > 0 \right\}.
\end{align*}
\end{prop}
\begin{proof}
Use Lemma 3.5 to deduce that the breakpoint for the value of $\Psi_k(r_{ij},\theta_i,\phi_j;\psi_\tau) = -r_{ij} \sin(\theta_i) \cos(\phi_j) x^2_k + g(a,b_k;\psi_\tau)$ for $a := r_{ij} \sin(\theta_i) (\cos(\phi_j) + \sin(\phi_j)),\; b_k := r_{ij} \sin(\theta_i) \cos(\phi_j) x_k - (\frac{r_{ij}}{2}) \cos(\theta_i),$ occurs when $(\tau - \frac{\tilde{b}_k}{\tilde{a}} ) = 0$. 
Note the functions $\Psi_k$ and hence $F$ are convex in $r$ for $(\theta_i,\phi_j)$ fixed. Also, under the constraint $r \sin(\theta) (\cos(\phi) + \sin(\phi)) > 1$, recall the dual problem (SD5) has finite value which implies $F$ has finite value below and hence a non-negative partial derivative $\partial_r F$ for some $r_{ij}$ with $(\theta_i,\phi_j)$ fixed. Note that $\partial_r F$ is continuous and only depends on $r_{ij}$ through the $K_1$ summand. It follows that there exists an extremal point $r^{*}_{ij}$ such that either $\partial_r F(r^*_{ij}) = 0$ or $r^*_{ij} \gtrsim 1/\sin(\theta_i) (\cos(\phi_j) + \sin(\phi_j))$ given $(\theta_i,\phi_j)$. 
Once again, concatentate a list $L$ of tuples $(\theta_i,\phi_j,r^{*}_{ij})$ under the constraint $\sin(\theta_i) (\cos(\phi_j) + \sin(\phi_j)) > 0$, evaluate $F(r^{*}_{ij},\theta_i,\phi_j;\psi_\tau)$, and return the $\min$.
\end{proof}

%

\begin{prop}
The solution to UFPM dual problem (SD6) can be computed (approximately) by redefining the set $K_1$, and using the approach described in Proposition 3.5.
In particular, the set $K_1$ is now defined as
\begin{align*}
K_1 &:= \left\{ k: \bigg( \frac{\tilde{b}_k}{\tilde{a}} - \tau \bigg) > 0 \right\}.
\end{align*}
\end{prop}
\begin{proof}
Use the new definition for the set $K_1$ and apply Lemma 3.6 instead of 3.5. Otherwise, follow the approach in the proof of Proposition 3.5.
\end{proof}

\endgroup
\section{Applications}
Let us now investigate some practical applications of the theory and algorithms developed in this work. A couple of examples from inventory control and financial markets are considered. In general, one can compare ``delta trajectories" to reach the classical limits, across moment problems and/or across data sets, on a relative basis. A mapping between Wasserstein distance $\delta$ and (statistical) confidence level $\beta = 0.95$ is done via the relation \ref{R1} given in Section 1.3.2. The DD and SM methods, from Sections 2.4 and 3.3, are used to evaluate $y_\theta$ and $y_{\{\theta_i,\phi_j\}}$ respectively. The algorithms are adapted to solve the particular moment problems of interest (e.g. zeroth partial moment (ZPM), first partial moment (FPM), and second partial moment (SPM), for lower and/or upper tail). The algorithms are coded in Matlab and make use of standard functions such as \textit{bisection}, \textit{fminbnd}, \textit{intersectLines}. No special Matlab toolboxes are needed (although parallel computing via \textit{parfor} loops for SM requires use of that toolbox). Results of the two methods are compared vs.\ each other and vs.\ the classical results (without the distributional ambiguity constraint) for consistency. It is useful to implement both methods to provide an additional check to the calculations. Note that SM gives results consistent with the DD method, within 1e-3 (on these examples), for a $(\theta,\phi)$ mesh of 750 by 750 grid points on the $[0,\pi] x [0,2 \pi)$ domain. As the mesh partition gets finer, the DD and SM results converge. 

\subsection{Two Point Example}
Consider a two point example with $x = \{10, 12\}$ which implies $\mu = 11$ and $\sigma = 1$. For the lower Chebyshev-Cantelli (LC) problem, with $\tau = \tau_1 = \mu$, the classical result is $\Pr(X \leq \tau) = 1$. For the upper Chebyshev-Cantelli (UC) problem, with $\tau = \tau_2 = \mu+1/2$, the classical result is $\Pr(X \geq \tau) = 0.8$. Let us investigate the solution trajectory for the robust problems as a function of distributional ambiguity, $\delta$.  For the LC problem, figure 1 shows a plot of the $U_{\{1,2\}}$ and $L_{\{1,2\}}$ lines across which the $\Psi_{\{1,2\}}$ functions and hence F change value. Figure 2 shows the solutions to the robust LC and UC problems, using the directional descent (DD) method, as a function of $\delta$. One can see that for the LC problem, the robust solution approaches the classical solution (CC) at $\delta^* \approx 2$. For the UC problem, the robust solution approaches the classical solution (CC) at $\delta^* \approx 1$. 

\begin{figure}[H]
	\centering
	\caption{DD Method: ZPM Plots for $\xi=1$}%
	\subfloat[U and L Lines]{\scalebox{0.2}[0.15]{\includegraphics{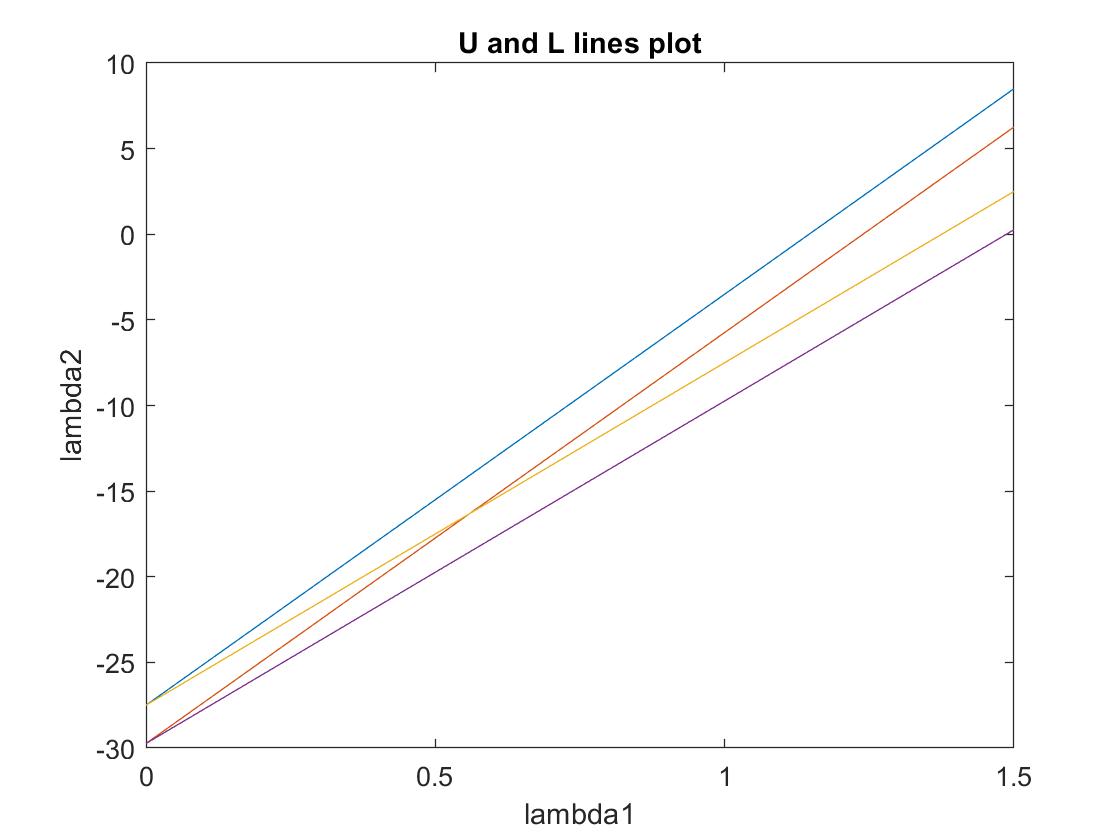}}}%
	\quad
	\subfloat[Surface Plot]{\scalebox{0.225}[0.15]{\includegraphics{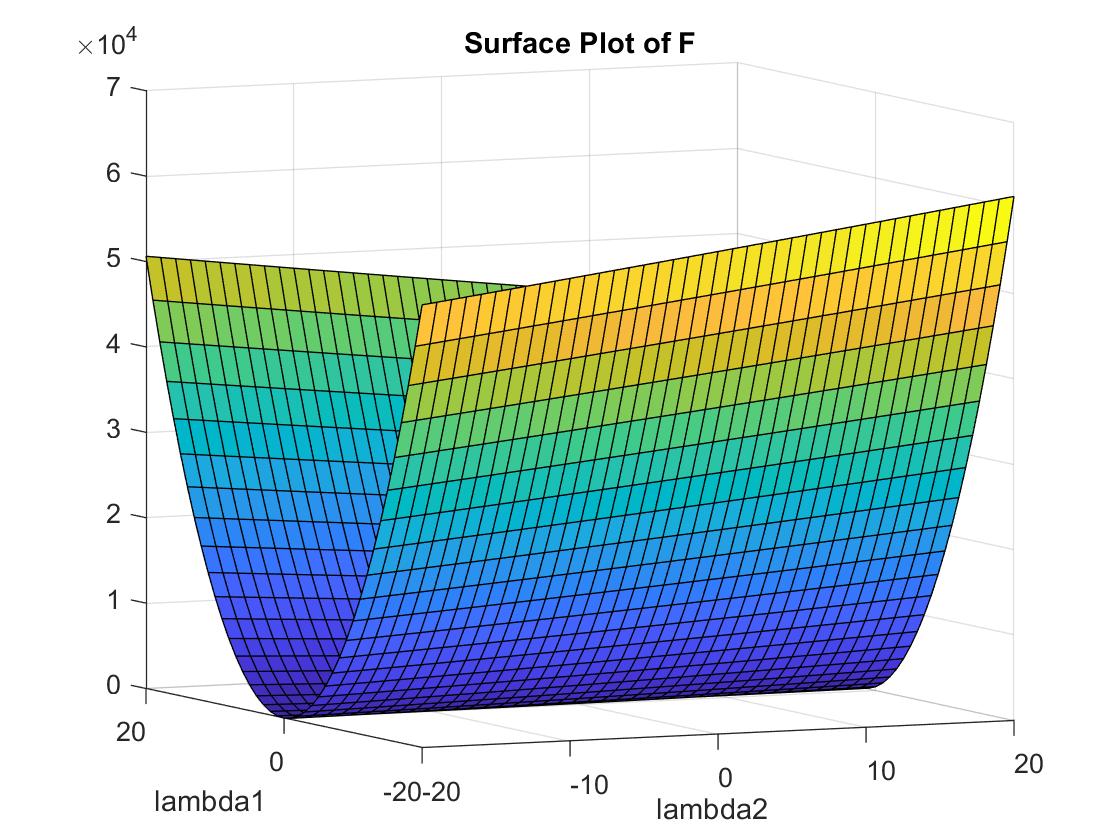}}}%
\end{figure}

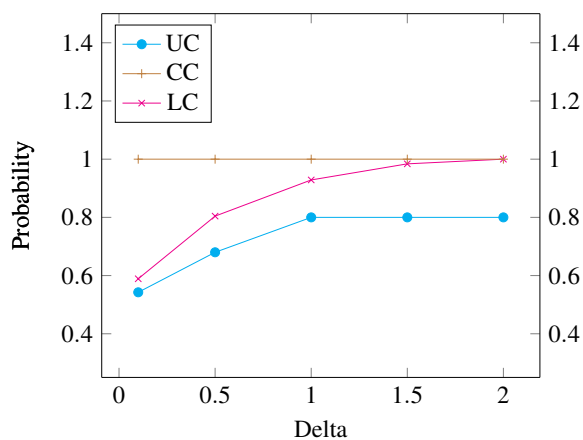
\begin{figure}[H]
\caption{Chebyshev-Cantelli Probabilities}
\begin{center}
\scalebox{0.85}{
\begin{tikzpicture}
	\begin{axis}[legend pos=north west,
		xlabel=Delta,
		ylabel=Probability,
		ymin = 0.25, ymax = 1.5,
		axis y line* = left ]
	\addplot[color=magenta,mark=x] coordinates {
		(0.1,0.589)
		(0.50,0.8043)
		(1.0,0.9286)
		(1.5,0.9839)
		(2.0,1.0)
	}; \label{plot1_y1}
	\end{axis}

	\begin{axis}[legend pos=north west,
		xlabel=Delta,
		ylabel=Probability,
		ymin = 0.25, ymax = 1.5,
		axis y line* = right,
		axis x line = none ]
	\addplot[color=cyan,mark=*] coordinates {
		(0.1,0.5425)
		(0.50,0.68)
		(1.0,0.8)
		(1.5,0.8)
		(2.0,0.80)
	}; \label{plot1_y2}

	\addplot[color=brown,mark=+] coordinates {
		(0.1,1.0)
		(0.50,1.0)
		(1.0,1.0)
		(1.5,1.0)
		(2.0,1.0)
	}; \label{plot1_y3}

    \addlegendimage{/pgfplots/refstyle=plot1_y1}\addlegendentry{UC}
	\addlegendimage{/pgfplots/refstyle=plot1_y2}\addlegendentry{CC}
	\addlegendimage{/pgfplots/refstyle=plot1_y3}\addlegendentry{LC}
	\end{axis}
\end{tikzpicture}
}
\end{center}
\end{figure}

\subsection{Inventory Control}

Consider the historical data set (in units of millions) in Table 4 for Apple iPhones sales, taken from the statista website \citep{statistaPhones}. Note that Apple stopped reporting iPhone sales in 2019. Let us investigate the robust probability and number of stockouts (lost sales). Here $\mu = 122.345, \sigma = 85.326$. Matlab calculates the quantiles as shown in Table 5. Let us perform a stockout analysis by setting $\tau = 221.77$ (the $90\textsuperscript{th}$ percentile) and calculating the robust upper zeroth and first partial moments. Figures 4 and 6 show the solutions using the DD algorithm. Results were cross-checked using SM. Using the empirical (reference) data set, the expected annual lost sales, for an order quantity $\tau$, is 0.7875 million units. Robust estimates for probability and number of stockouts at $\beta = 0.95$ (which corresponds to $\delta \approx 290$ via \ref{R1} with $r \approx 231$) would be 38\% and 4.45 million units respectively. The classical limits are 42.4\% and 15.8 million units. To go further, one could extend our framework to construct worst case distributions, as a function of $\delta$, to ``back out" the sales distributions that give rise to the corresponding level of stockouts given by the solution to the dual problem (D4). See \citet{singh2020robust} for further details. 

\begin{table}[!htb]
\begin{center}
\caption{Apple iPhone Historical Sales (Worldwide)}
\begin{tabular}{ |c|c|c|c|c|c|c|c|c|c|c|c|c| }
 \hline
Year & 2007 & 2008 & 2009 & 2010 & 2011 & 2012 & 2013 & 2014 & 2015 & 2016 & 2017 & 2018 \\
 \hline
Sales & 1.39 & 11.63 & 20.73 & 39.99 & 72.29 & 125.05 & 150.26 & 169.22 & 231.22 & 211.88 & 216.76 & 217.72 \\
 \hline
\end{tabular}
\end{center}
\end{table}

\begin{table}[!htb]
\begin{center}
\caption{Quantiles: Apple iPhone Sales}
\begin{tabular}{ |c|c|c|c|c|c|c|c|c|c|c|c|c| }
 \hline
Year & 0.1 & 0.2 & 0.3 & 0.4 & 0.5 & 0.6 & 0.7 & 0.8  & 0.9 & 1.0 \\
 \hline
Sales & 8.558 & 19.82 & 43.22 & 88.118 & 137.655 & 163.532 & 207.614 & 216.856 & 221.77 & 231.22 \\
 \hline
\end{tabular}
\end{center}
\end{table}

\begin{figure}[!htb]
	\centering
	\caption{DD Method: ZPM Plots for $\xi=1$}%
	\subfloat{\scalebox{0.2}[0.2]{\includegraphics{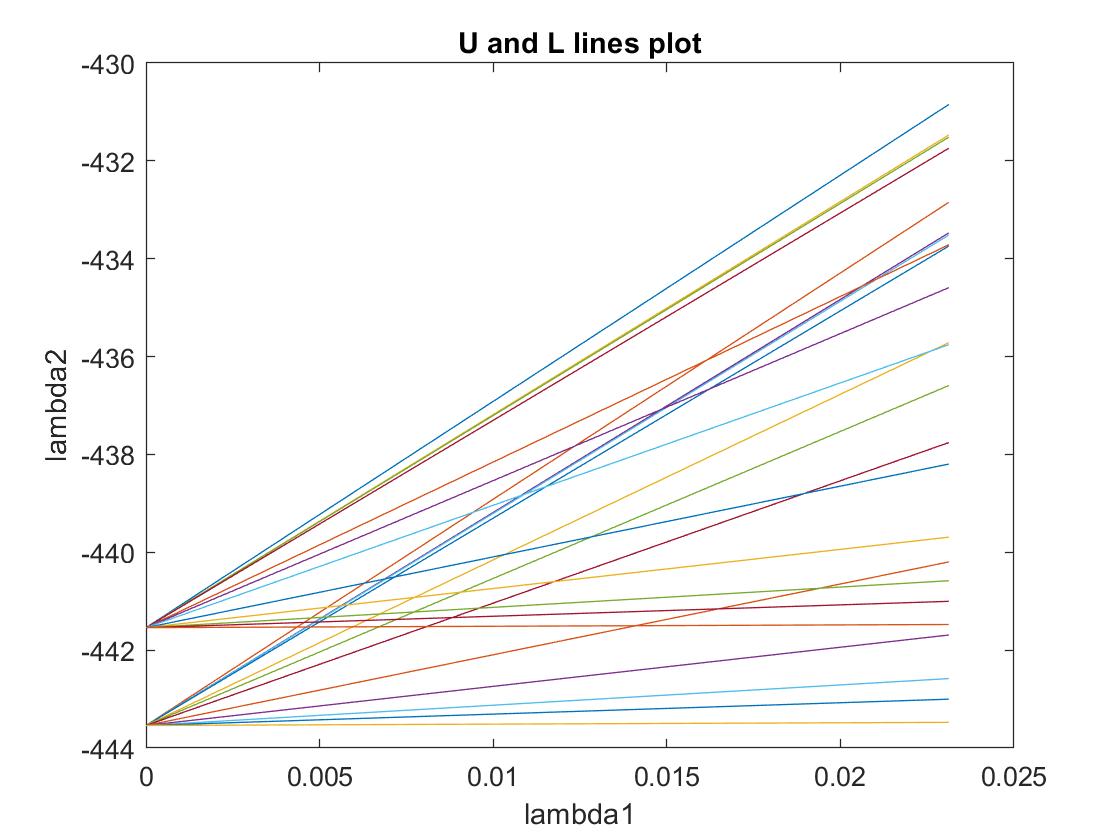}}}%
	\quad
	\subfloat{\scalebox{0.225}[0.2]{\includegraphics{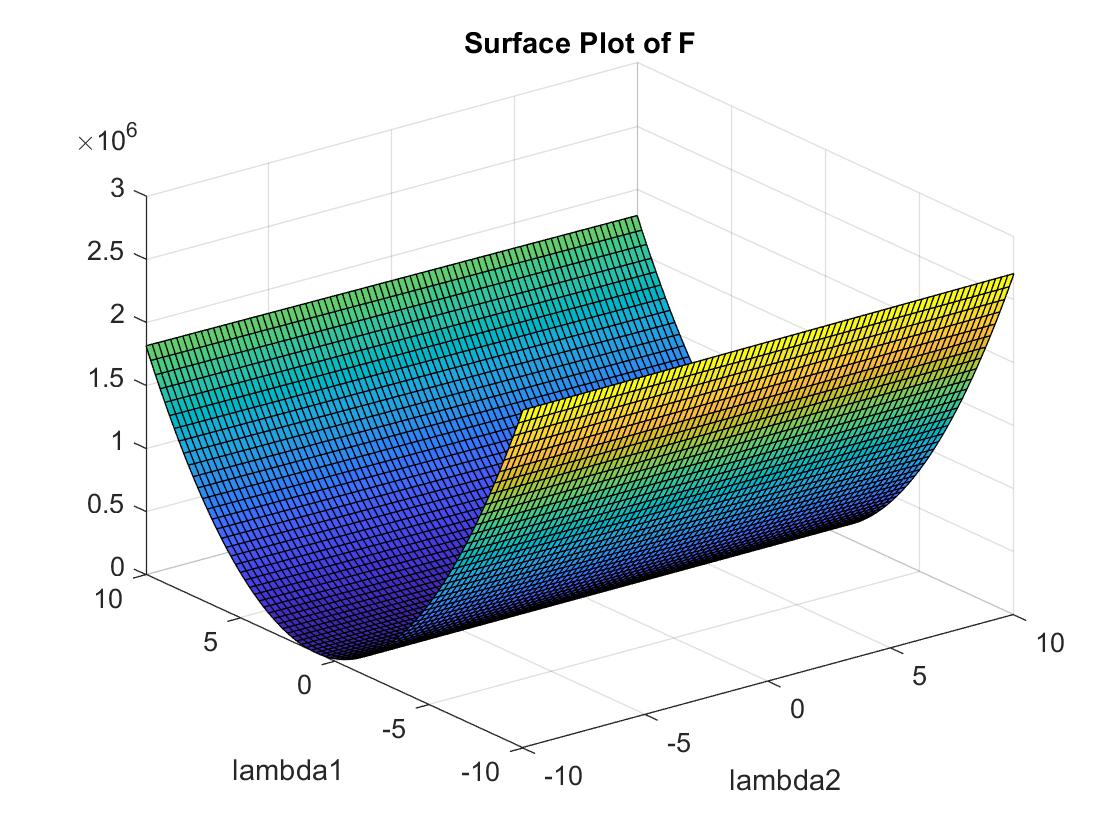}}}%
\end{figure}

\begin{figure}[H]
\caption{Probability of Stockout at $90\textsuperscript{th}$ Percentile }
\begin{center}
\scalebox{0.9}{
\begin{tikzpicture}[scale=0.975]
	\begin{axis}[legend pos=north west,
		xlabel=Delta (thousands),
		ylabel=Lost Sales (mn),
		ymin = 0.0, ymax = 0.60 ]
	\addplot[color=brown,mark=x] coordinates {
		(.001,0.1368)
		(.005,0.2402)
		(.010,0.2751)
		(.025,0.3334)
		(.100,0.3518)
		(.250,0.3762)
		(.500,0.3984)
		(.600,0.4040)
		(.800,0.4119)
		(.900,0.4147)
		(1.000,0.417)
		(1.200,0.42)
		(1.650,0.424)
	}; \label{plot12_y1}

	\addplot[color=green,mark=x] coordinates {
		(.001,0.424)
		(.100,0.424)
		(.250,0.424)
		(.500,0.424)
		(.600,0.424)
		(.800,0.424)
		(.900,0.424)
		(1.000,0.424)
		(1.200,0.424)
		(1.650,0.424)
	}; \label{plot12_y2}

	\addlegendimage{/pgfplots/refstyle=plot12_y1}\addlegendentry{Stockouts}
	\addlegendimage{/pgfplots/refstyle=plot12_y2}\addlegendentry{Classical}
	\end{axis}	
\end{tikzpicture}
}
\end{center}
\end{figure}
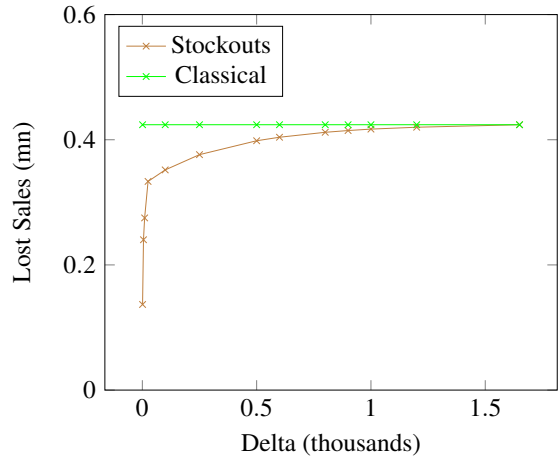

\begin{figure}[!htb]
	\centering
	\caption{DD Method: FPM Plots for $\xi=1$}%
	\subfloat{\scalebox{0.2}[0.2]{\includegraphics{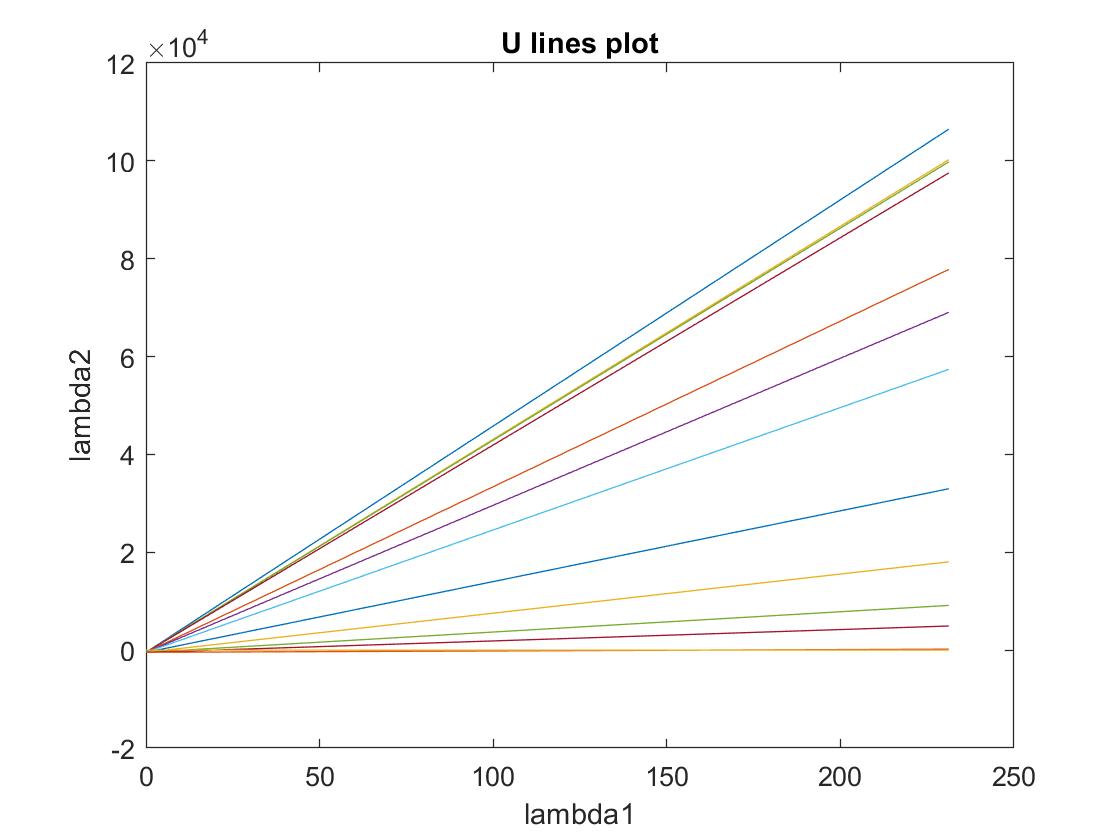}}}%
	\quad
	\subfloat{\scalebox{0.225}[0.2]{\includegraphics{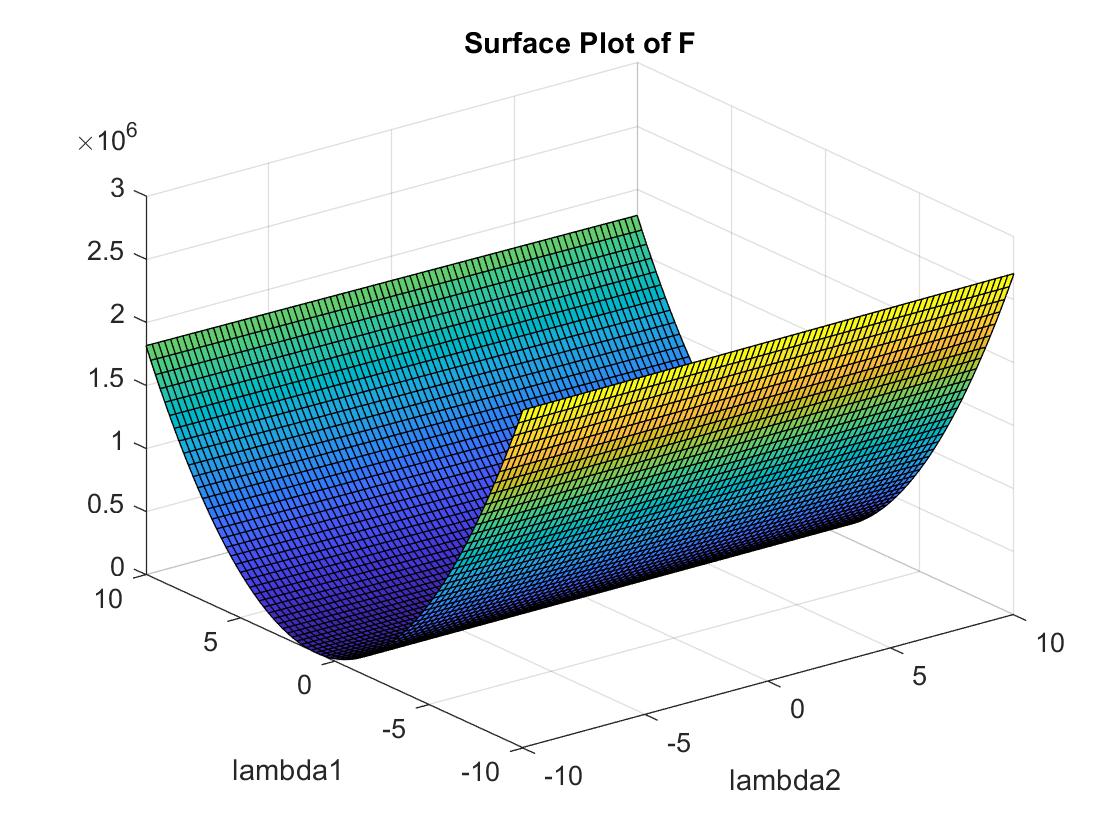}}}%
\end{figure}

\begin{figure}[H]
\caption{Number of Stockouts at $90\textsuperscript{th}$ Percentile }
\begin{center}
\scalebox{0.9}{
\begin{tikzpicture}[scale=0.975]
	\begin{axis}[legend pos=north west,
		xlabel=Delta (thousands),
		ylabel=Lost Sales (mn),
		ymin = 0.0, ymax = 25.0 ]
	\addplot[color=brown,mark=x] coordinates {
		(0.001,1.05)
		(0.010,1.59)
		(0.100,3.485)
		(0.500,6.844)
		(1,9.06)
		(2,11.98)
		(3,13.35)
		(4,14.38)
		(4.5,14.80)
		(5,15.24)
		(6,15.80)
	}; \label{plot2_y1}

	\addplot[color=green,mark=x] coordinates {
		(0.001,15.8)
		(0.010,15.8)
		(0.100,15.8)
		(0.500,15.8)
		(1,15.8)
		(2,15.8)
		(3,15.8)
		(4,15.80)
		(4.5,15.80)
		(5,15.80)
		(6,15.80)
	}; \label{plot2_y2}

	\addlegendimage{/pgfplots/refstyle=plot2_y1}\addlegendentry{Stockouts}
	\addlegendimage{/pgfplots/refstyle=plot2_y2}\addlegendentry{Classical}
	\end{axis}
\end{tikzpicture}
}
\end{center}
\end{figure}
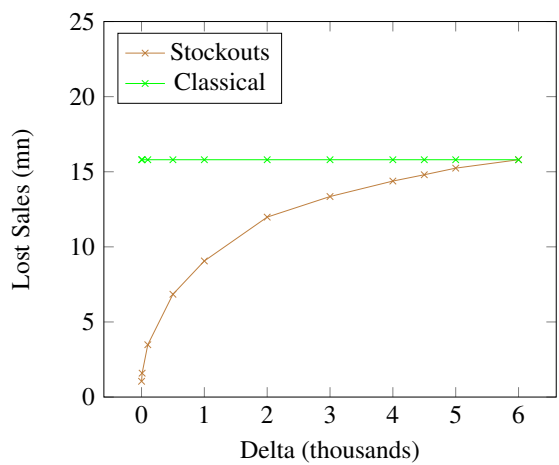

\subsection{Investment Portfolio}
Basket trading involves simultaneous trading of a basket of stocks. This example looks at the trajectory of semideviation of monthly portfolio returns for a small, equal dollar-weighted ``buy and hold" basket of U.S. equities from the S\&P 500 index used in the statistical arbitrage study by \citep{singh2020robust}. Data is sourced from the Yahoo finance website. Table 6 below lists the stock tickers, names, and industries. Table 7 displays a \textit{partial} listing of the 5y historical market data set (of 60 month end stock prices) from September 2015 through September 2020 used in this study. For convenience, and ease of interpretation, the monthly portfolio returns are scaled up by a factor of 100, so that a 1\% return is recorded in the data set as 1 instead of 0.01. This rescaling of the data only affects interpretation of the ambiguity parameter, $\delta$. For both the lower and upper second partial moment problems, (LSPM and USPM respectively), we set $\tau = \mu = 1.11\%$; note that $\sigma = 9.43\%$ for this data set. Figure 7 shows the solutions using SM algorithm, as a function of $\delta$. Results were cross-checked using the DD method. A robust number for lower semideviation at $\beta = 0.95$ (which corresponds to $\delta \approx 15.4$ via \ref{R1} with $r \approx 30\%$) would be 8.13\%.
For LSPM, the lower semideviation result approaches the classical solution (CSPM) of $\sigma = 9.43\%$ at $\delta^* \approx 175$. A robust number for upper semideviation at $\beta = 0.95$ (which corresponds to $\delta \approx 22.1$ via \ref{R1} with $r \approx 43.3\%$) would be 8.8\%. 
For USPM, the upper semideviation result approaches the classical solution (CSPM) of $\sigma = 9.43\%$ at $\delta^* \approx 160$. Both trajectories start to flatten out around $\delta = 75$, which corresponds to semideviation $\approx 9.2\%$. One can infer that semideviation beyond this is remote; both of these curves exhibit a ``long right tail" tail that slowly converges to the classical solution. More precisely, for LSPM $\delta \approx 24.6$ corresponds to $\beta = 0.999$ and for USPM $\delta \approx 35.3$ corresponds to $\beta = 0.999$.

\begin{table}[!htb]
\begin{center}
\caption{Basket Constituents}
\begin{tabular}{ |c|c|c|c| }
 \hline
  Ticker & Name & Industry & Market Cap (bn) \\
 \hline
 APA & Apache Corporation & Energy: Oil and Gas & 4.11 \\
 \hline
 AXP & American Express Company & Credit Services & 76.80 \\
 \hline
 CAT & Caterpillar Inc. & Farm Machinery & 78.60 \\
 \hline
 COF & Capital One Financial Corp. & Credit Services & 31.10 \\
 \hline
 FCX & Freeport-McMoRan Inc. & Copper & 22.35 \\
 \hline
 IBM & 1nternational Business Machines Corp. & Technology & 105.17 \\
 \hline
 MMM & 3M Company & Industrial Machinery & 92.41 \\
 \hline
\end{tabular}
\end{center}
\end{table}

\begin{table}[H]
\begin{center}
\caption{Basket 2020 Market Data}
\begin{tabular}{ |r|r|r|r|r|r|r|r|r|r| }
 \hline
Date & 01/01 & 02/01 & 03/01 & 04/01 & 05/01 & 06/01 & 07/01 & 08/01 & 09/01 \\
 \hline
 APA & 27.10 & 24.80 & 4.16 & 13.02 & 10.77 & 13.48 & 15.32 & 14.8 & 11.1 \\
 \hline
 AXP & 128.13 & 108.83 & 84.75 & 90.34 & 94.64 & 94.77 & 92.90 & 101.60 & 95.56 \\
 \hline 
 CAT & 128.29 & 122.19 & 114.13 & 114.46 & 119.24 & 125.56 & 131.89 & 142.31 & 144.38 \\
 \hline
 COF & 98.61 & 87.20 & 50.01 & 64.23 & 67.49 & 62.50 & 63.70 & 68.92 & 68.13 \\
 \hline
 FCX & 11.06 & 9.96 & 6.75 & 8.83 & 9.07 & 11.57 & 12.92 & 15.61 & 14.77 \\
 \hline
 IBM & 138.55 & 125.46 & 108.05 & 122.30 & 121.65 & 119.21 & 121.35 & 121.72 & 118.83 \\
 \hline
 MMM & 154.29 & 145.13 & 133.95 & 149.07 & 153.51 & 154.58 & 149.11 & 161.55 & 159.51 \\
 \hline
\end{tabular}
\end{center}
\end{table}

\begin{figure}[H]
	\centering
	\caption{DD Method: SPM Plots for $\xi=1$}%
	\subfloat{\scalebox{0.2}[0.25]{\includegraphics{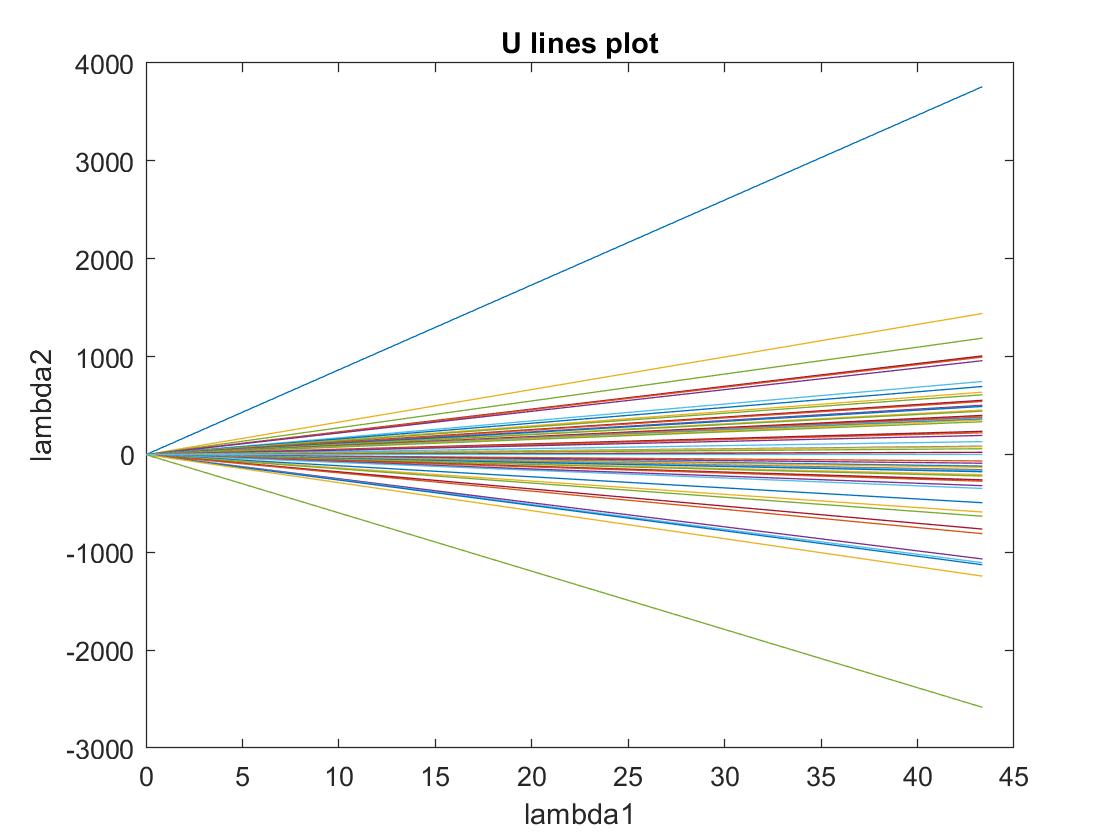}}}%
	\quad
	\subfloat{\scalebox{0.225}[0.25]{\includegraphics{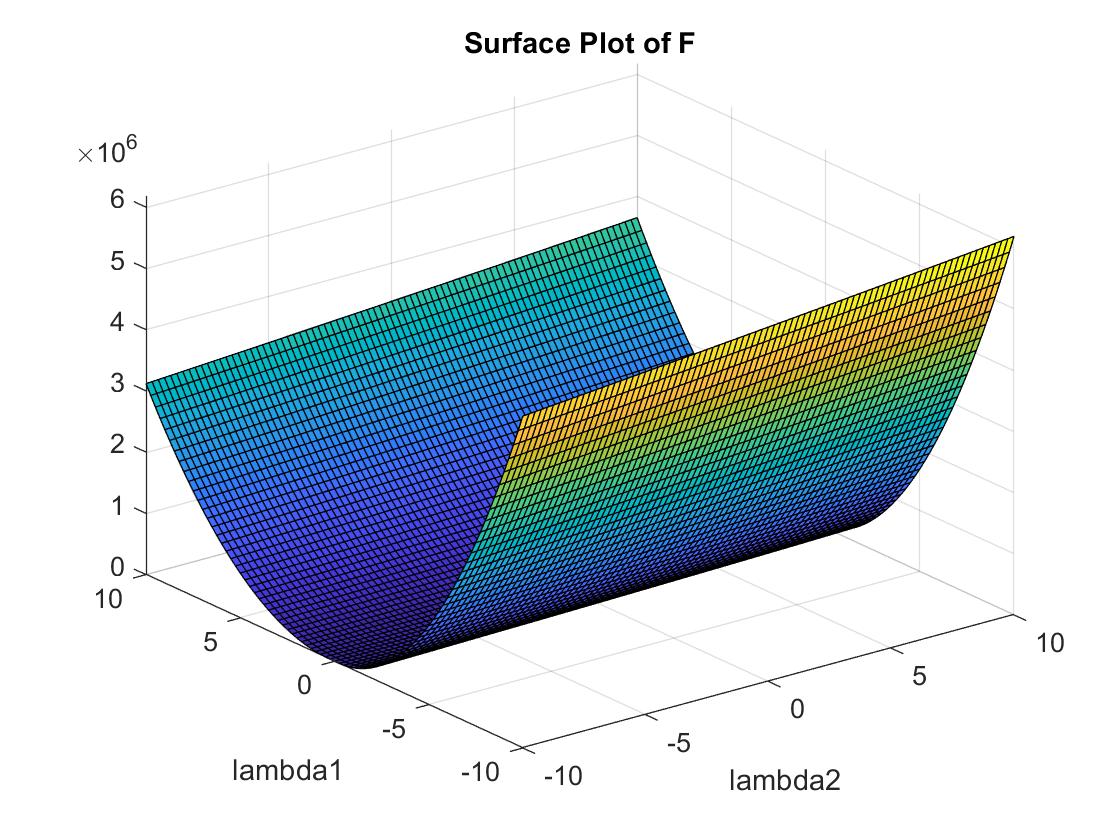}}}%
\end{figure}

\begin{figure}[!htb]
	\centering
	\caption{Investment Portfolio}%
	\subfloat{\scalebox{0.70}[0.70]{\includegraphics{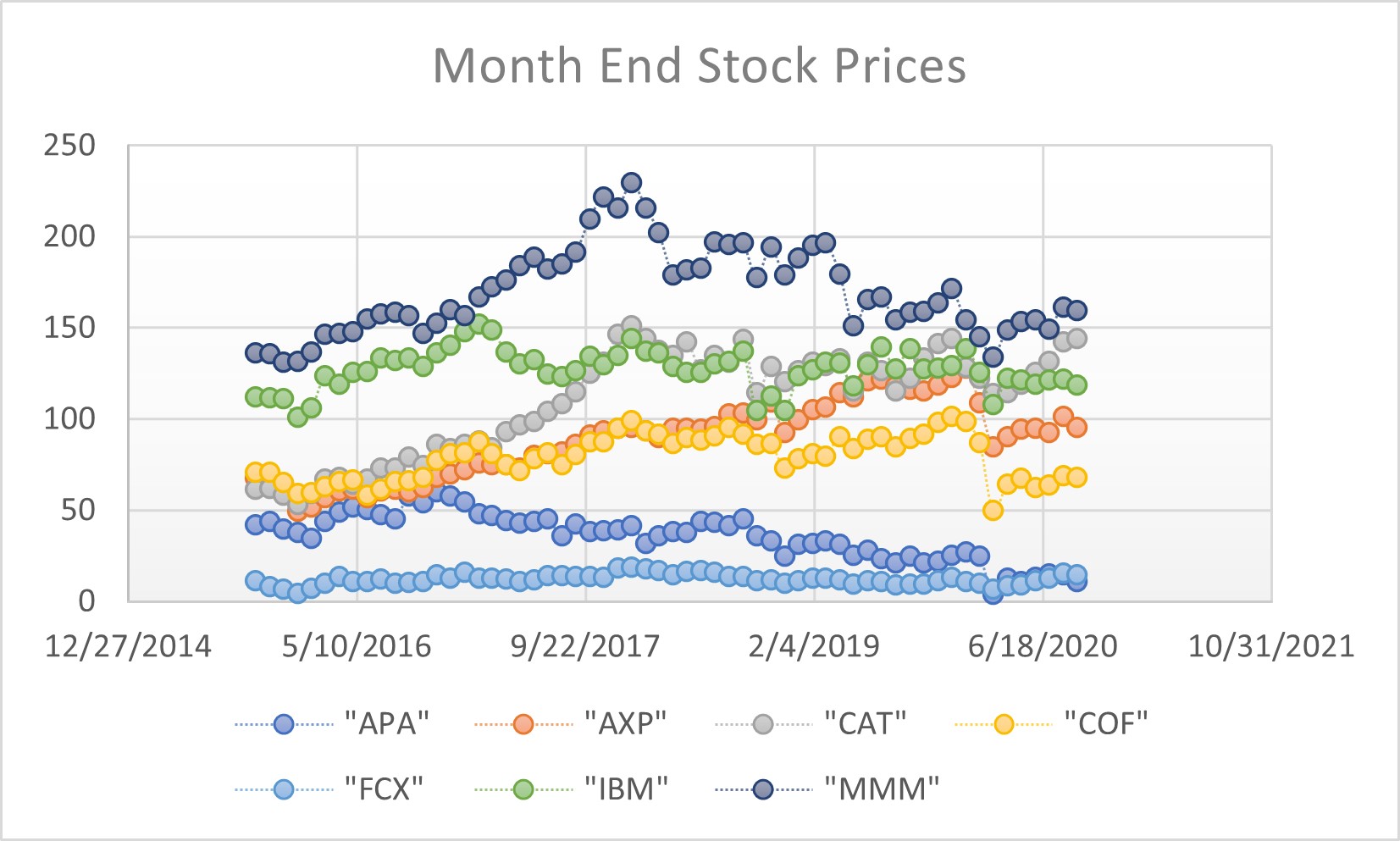}}}%
	\quad
	\subfloat{\scalebox{0.70}[0.70]{\includegraphics{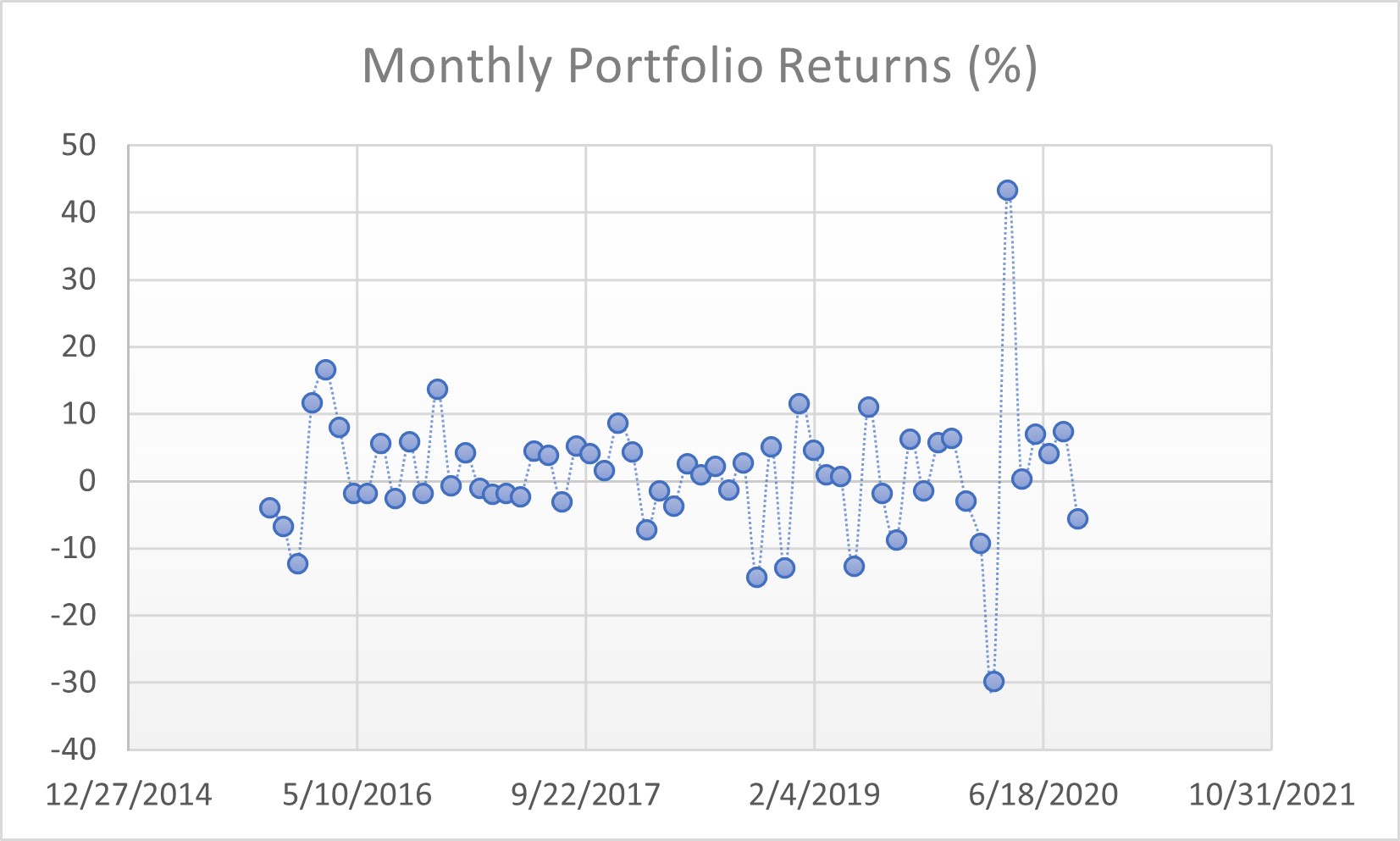}}}%
\end{figure}

\begin{figure}[H]
\caption{Semideviation of Portfolio Returns (\%)}
\begin{center}
\scalebox{0.9}{
\begin{tikzpicture}
	\begin{axis}[legend pos=north west,
		xlabel=Delta,
		ylabel=Semideviation (\%),
		ymin = 0, ymax = 16,
		axis y line* = left ]
	\addplot[color=magenta,mark=x] coordinates {
		(1,6.82)
		(10,7.83)
		(30,8.63)
		(60,9.12)
		(100,9.35)
		(150,9.42)
		(175,9.428)
	}; \label{plot3_y1}
	\end{axis}

	\begin{axis}[legend pos=north west,
		xlabel=Delta,
		ylabel=Semideviation (\%),
		ymin = 0, ymax = 16,
		axis y line* = right,
		axis x line = none ]
	\addplot[color=cyan,mark=*] coordinates {
		(1,7.49)
		(10,8.34)
		(30,8.98)
		(60,9.31)
		(100,9.39)
		(150,9.42)
		(160,9.426)
	}; \label{plot3_y2}

	\addplot[color=green,mark=+] coordinates {
		(1,9.43)
		(10,9.43)
		(30,9.43)
		(60,9.43)
		(100,9.43)
		(150,9.43)
		(160,9.43)
	}; \label{plot3_y3}

    \addlegendimage{/pgfplots/refstyle=plot3_y1}\addlegendentry{USPM}
	\addlegendimage{/pgfplots/refstyle=plot3_y2}\addlegendentry{CSPM}
	\addlegendimage{/pgfplots/refstyle=plot3_y3}\addlegendentry{LSPM}
	\end{axis}
\end{tikzpicture}
}
\end{center}
\end{figure}
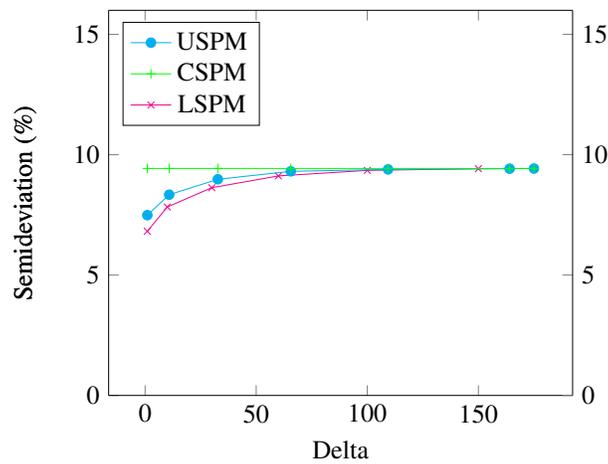

\section{Conclusions and Further Work}
This work has developed theoretical results and investigated calculations of univariate DRMPs using Wasserstein distance as an ambiguity measure. The moments problem overview and foundational notation and problem definitions were introduced in Section 1. Using problem of moments duality results, the simpler dual formulation and its mixture of analytic and computational solutions were derived in Section 2. In Section 3, we developed a computational approach (the spherical method) to solve these DRMPs in a simpler way. In Section 4, we applied our results to particular problem instances in inventory control and option pricing (univariate setting). Finally, we conclude with some commentary on directions for further research. \par
One direction for future research would be to investigate DRMPs in a multivariate setting, using the tools of SDP. Another direction for future research would be to extend the methods developed in Sections 2 and 3 to address additional moments problems beyond Cheybyshev-Cantelli or the first two partial moments. Finally, perhaps a third direction for future research would be to investigate extensions of the distributionally robust framework to compute worst case distributions and/or incorporate a decision problem. \par

\section*{Data and Code Availability Statement}
The raw and/or processed data, as well as the Matlab code, required to reproduce the findings from this research can be obtained from the corresponding author, [D.S.], upon reasonable request.

\section*{Conflict of Interest Statement}
The authors declare they have no conflict of interest.

\section*{Funding Statement}
The authors received no specific funding for this work.

\bibliographystyle{apalike}
\bibliography{DRMP}

\end{document}